\documentclass[reqno]{amsart}
\usepackage[utf8]{inputenc}
\usepackage[english]{babel}
\usepackage{latexsym,amssymb,amsthm,xcolor}
\usepackage[centertags]{amsmath}
\usepackage{mathtools}
\usepackage{scrtime}
\usepackage{graphicx}
\usepackage{bbm}
\usepackage{framed}
\usepackage{fancyhdr}
\usepackage{appendix}   
\usepackage{enumerate}
\usepackage{tikz} 
\usepackage{pgfplots}
\usepackage{pgffor}
\usepackage{hyperref}
\usepackage{accents}           % see http://tex.stackexchange.com/a/125414/16865
\usepackage{microtype}

\makeatletter
\renewcommand{\p@enumii}{}% Remove prefix for second-level enumerate item
\makeatother

\mathtoolsset{showonlyrefs=true}

\newtheorem{proposition}{Proposition}[section]
\newtheorem{corollary}[proposition]{Corollary}
\newtheorem{lemma}[proposition]{Lemma}
\newtheorem{definition}[proposition]{Definition}
\newtheorem{theorem}[proposition]{Theorem}

\makeatletter                                           % This is to ensure that remark titles are 
\def\th@newremark{\th@remark\thm@headfont{\bfseries}}   % style) but the body is in normal font 
\makeatletter                                           % in theorem style). The standard remark 
\theoremstyle{newremark}
\newtheorem{remark}{Remark}[section]

\numberwithin{equation}{section}

\newcommand{\eps}{\varepsilon}
\newcommand{\dx}{\mathrm{d}}                                   % the d for differentials
\newcommand{\eqd}{\overset{\textup{d}}{=}}                     % *e*qual *i*n *d*istribution
\DeclareMathOperator*{\wlim}{w-lim}

\allowdisplaybreaks[4]

\newcommand{\mcA}{\mathcal{A}}

\newcommand{\mcD}{\mathcal{D}}

\newcommand{\mcF}{\mathcal{F}}

\newcommand{\mcQ}{\mathcal{Q}}

\newcommand{\mcZ}{\mathcal{Z}}
\newcommand{\bbE}{\mathbb{E}}

\newcommand{\bbN}{\mathbb{N}}

\newcommand{\bbP}{\mathbb{P}}

\newcommand{\bbR}{\mathbb{R}}

\newcommand{\bbZ}{\mathbb{Z}}
\begin{document}

\title{Longtime behavior of completely positively correlated Symbiotic Branching 
Model}
\author{Patric Karl Glöde, Leonid Mytnik}
\address{Technion---Israel Institute of Technology, Faculty of Industrial 
Engineering \& Management, Haifa 32000, Israel}
\email{Patric.Gloede@t-online.de, leonid@ie.technion.ac.il}
\date{September 2022}
\thanks{P.K. Glöde  was partially supported by the Technion Fund; both authors were partially supported by German Israeli Foundation grant 1170-186.6/2011 and L. Mytnik was partially supported by Israel Science Foundation Grant  1704/18}
% and 1985/22}
\begin{abstract}
We study the longtime behavior of  a continuous state Symbiotic Branching 
Model (SBM). SBM can be seen as a unified model generalizing the Stepping 
Stone Model, Mutually Catalytic Branching Processes, and the Parabolic Anderson 
Model. It was introduced by Etheridge and Fleischmann~\cite{EF04}. The 
key parameter in these models is the local correlation $\rho$ between the 
driving Brownian Motions.  The longtime behavior of all SBM exhibits a dichotomy 
between coexistence and non-coexistence of the two populations depending 
on the recurrence and transience of the migration and also in many cases on the branching rate. The most 
significant gap in the understanding of the longtime behavior of SBM is for 
positive correlations in the transient regime. In this article we give a precise 
description of the longtime behavior of the SBM  with $\rho=1$ with not necessarily identical initial 
conditions.% and for $\rho$ sufficiently close to $1$.
\end{abstract}

\maketitle

\section{Introduction and Results}

\subsection{Symbiotic Branching Dynamics}

The Symbiotic Branching Model (SBM)  is a continuous state model of a two type 
branching population living and migrating in a geographic space. The 
branching rate of each of the two populations locally depends on the 
size of the other population. While the SBM is a continuous state model the 
dynamics are best understood by having an informal look at the discrete 
particle approximation. Both the continuous and the discrete model depend 
essentially on the branching rate $b>0$ and the correlation parameter 
$\rho\in[-1,1]$. Geographic space is represented by $\bbZ^d$. So the model 
describes the evolution in time of the number of particles of each population 
at each site $k\in\bbZ^d$ and let us denote the process by 
$(X,Y)=(X_t,Y_t)_{t\geq0}$. We refer to the particles of the two populations as 
particles of type $1$ and $2$, respectively. The evolution of the particle model is as follows: 

\begin{itemize}
 \item At each site, each pair of particles of opposite type share an 
exponential clock of rate $b |\rho|$. When the clock rings, both particles 
die and the following happens depending on the parameter $\rho$:

If $\rho\in[-1,0)$, with equal probability either the type 1 particle 
has two offspring and the type 2 particle has no offspring or vice 
versa.

If $\rho\in(0,1]$, with equal probability either both particles have zero 
or two offspring.

 \item At time $t>0$, each particle of type 1 living at site $k$ has an 
exponential clock of rate $b (1-|\rho|) Y_t(k)$, each particle of type 2 at site 
$k$ has an 
exponential clock of rate  $b (1-|\rho|) X_t(k)$. When the clock rings, the 
particle dies and with equal probability has either zero or two offspring.
 \item All particles perform simple continuous, rate $1$, random walks independently from each other.
 \item The migration and the branching dynamics are independent from each other.
\end{itemize}

The particle system described above has a scaling limit. It can be shown that 
if each particle is assigned mass $1/n$ and the initial mass is of order $n$,
then the rescaled particle system converges to a system of interacting 
diffusions $(u,v)=(u_t,v_t)_{t\geq0}$ which satisfy the following stochastic 
differential equations:
\begin{equation}
\begin{array}{l}
 \dx u_t(i)=\Delta u_t(i)\,\dx t+\sqrt{b u_t(i)v_t(i)}\,\dx W^u_t(i)\,,\quad t\geq 0,\\[1ex]
 \dx v_t(i)=\Delta v_t(i)\,\dx t+\sqrt{b u_t(i)v_t(i)}\,\dx W^v_t(i)\,, \quad t\geq 0, 
 \end{array}
\end{equation}
where $i\in\bbZ^d$ and $\{(W^u(i),W^v(i)):i\in\bbZ^d\}$ is an independent field 
of locally correlated planar Brownian motions and the spatial correlation is given by 
$\rho$. This model was introduced Etheridge and Fleischmann \cite{EF04} and is 
known as the Symbiotic Branching Model (SBM).

The SBM generalizes a couple of famous particle models. If $\rho=-1$, the 
Brownian motions $W^u(i)$ and $W^v(i)$  are 
totally anti-correlated for each $i\in \bbZ^d$. Under the additional condition that 
$u_0+v_0=\boldsymbol{1}$ for all 
$k\in\bbZ^d$, the process $u_t+v_t$ solves the heat equation and hence since the 
initial conditions 
are constant, $u_t+v_t\equiv1$ for all $t\geq0$. Therefore $v_t=1-u_t$ for all 
$t\geq0$. One can therefore rewrite \eqref{SBM} to obtain the well known system 
of differential equations for the 
stepping stone model (also known as interacting Fisher-Wright diffusions), see 
Shiga \cite{S80}. 
If $\rho=0$, the Brownians motions $W^u(i)$ and $W^v(i)$ are independent for each $i\in \bbZ^d$. This leads to the 
well known mutually catalytic branching model, studied by Dawson and Perkins 
\cite{DP98}, Cox, Dawson and Greven~\cite{bib:cdg04},  Cox, Klenke and Perkins~\cite{CKP00} among others. In continuous space, in dimension $d=2$, this models was studied in a series of papers by  Dawson et. al.~\cite{bib:defmpx02a}, \cite{bib:defmpx02b}, \cite{bib:dfmpx03}.
%In related works by Klenke and Mytnik~\cite{bib:km10}, \cite{bib:km12a}, \cite{bib:km12b} introduced the so called  infinite rate mutually catalytic model which was later extended
%to models with correlation (see ...
%\cite{bib:km20}\ldots)
If $\rho=1$, the Brownian motions $W^u(i)$ and $W^v(i)$ are totally positively correllated  for each $i\in \bbZ^d$. Under the additional 
assumption that the identical initial conditions are the same for both 
populations, that is, $u_0=v_0$, the SBM coincides with the Parabolic Anderson Model 
(PAM). For more information about the PAM see Carmona and Molchanov 
\cite{CM94}, Greven and den Hollander \cite{GDH07}, 
Gl\"{o}de \cite{G06}.

Let us note, that after certain limiting procedures, the above models give a rise to the so called infinite rate mutually catalytic and symbiotic models, 
that were studied extensively in the recent years, see Klenke and Mytnik~\cite{bib:km10}, \cite{bib:km12a}, \cite{bib:km12b}, \cite{bib:km20}, 
Blath, Hammer and Ortgiese~\cite{bib:bho16}, Hammer, Ortgiese and Florian~\cite{bib:hov18}.

In our paper we are interested in the longtime behavior of the SBM with $\rho=1$ in the case of not necessarily equal initial conditions --- the case that has not been studied in the literature.
 The main 
question is whether both populations can survive forever or whether just one 
population will survive while the other population will die out. If there is a 
positive probability that both populations will survive we will say coexistence 
is possible. Otherwise we will say coexistence is impossible.

The longtime behavior of the SBM has been thoroughly studied for different correlation parameters $\rho$. 
%It turns out that for the SBM the longtime behavior depends on the correlation parameter. 
For correlations $\rho\in(-1,0]$ it 
has been proved that there is a clear dichotomy: 
coexistence is possible if and only if the migration is transient. See 
Blath, D\"{o}ring and  Etheridge \cite{BDE11}, Dawson and Perkins \cite{DP98}, and 
D\"{o}ring and Mytnik \cite{DM13}. %{MR3111225}.
 For $\rho=-1$, Shiga \cite{S80} has 
proved that in the recurrent regime coexistence is impossible for the particular case of $u_0+v_0=1$. 
For correlation $\rho\in(0,1)$, Blath, D\"{o}ring, and Etheridge \cite{BDE11} 
have proved that if the migration is recurrent, coexistence is impossible. 
%It is believed that for positive correlation there is not a clear 
%transience/recurrence dichotomy as for $\rho\in(-1,0]$ but that as in the case 
%of the PAM there is a critical branching parameter such for all $b$ larger than 
%this parameter coexistence is impossible also in the transient regime. This 
%conjecture however has not yet been proved.

For $\rho=1$ in the case of identical initial conditions (the PAM) the phase 
transition between survival and non-survival (note that in the PAM there is only 
one population so it does not make sense to speak of coexistence of two 
populations) occurs in the transient regime, that is, survival  is impossible if 
the migration is recurrent. If the migration is transient then there is a 
critical branching parameter $b_\ast$ such that for $b>b_\ast$ survival is 
impossible while for $b<b_\ast$ survival is possible. See Theorem~\ref{t:DGH} below (which is essentially the result of Greven and den 
Hollander \cite{GDH07}). 

For correlation $\rho\in(0,1)$, Blath, D\"{o}ring, and Etheridge \cite{BDE11} 
have proved that if the migration is recurrent, coexistence is impossible. 
It is believed that for positive correlation %there is not a clear 
%transience/recurrence dichotomy as for $\rho\in(-1,0]$ but that as in the case 
%of the PAM
there is a critical branching parameter such for all $b$ larger than 
this parameter coexistence is impossible also in the transient regime (like in the case of  the PAM). This 
conjecture however has not yet been proved.
% In fact 
%D\"{o}ring and Mytnik \cite{DM13} have a very simple but unpublished proof that 
%for $\rho\in(0,1]$, for very small branching parameters $b$ coexistence is 
%possible. The threshold parameter they find however is clearly too small. 

Thus, note that for $\rho\in(-1, 0]$ and for the PAM 
the longtime behavior is fully characterized. Also, in the 
recurrent regime the longtime behavior of the SBM is fully understood for all 
$\rho$. However, in the transient regime there are gaps in the understanding of 
what happens for $\rho=-1$, with
$u_0+v_0\neq 1$,  for $\rho\in (0,1)$, and for $\rho=1$ with non-identical initial conditions.  In this paper our aim is to 
contribute towards a more complete understanding of the longtime behavior for 
positive correlations.

For $\rho\in (0,1]$ one of the main open questions related to the longtime behavior of SBM  is:
\begin{itemize}
 \item[] If the migration is transient, is there a critical parameter 
$b_\#(\rho)$ such that for $b<b_\#(\rho)$ coexistence is possible while for 
$b>b_\#(\rho)$ it is impossible.
\end{itemize}

Our main result, Theorem \ref{t:main} gives a partial answer to this 
question: we characterize coexistence/non-coexistence dichotomy for SBM for the case of  $\rho=1$ and 
with initial conditions that are not necessarily equal. 
%$u_0\neq v_0$ we can give a complete description  
%of the longtime behavior.

The paper is organised as follows. In Section 
\ref{s:definitions} we formally introduce the SBM and related concepts that will be investigated.  
Section \ref{s:overview} states  existing relevant
results on the longtime behavior of the PAM which is of profound importance for studying the long time behavior of the SBM with correlation $\rho=1$. In Section \ref{s:mainresult} we 
present our main result for the longtime behavior of the SBM with correlation 
$\rho=1$. The proof of our result is split into two parts. 
In Section \ref{s:proofcoex} we treat the regime when the coexistence is possible while in Section 
\ref{s:proofnoncoex} we deal with the regime when the coexistence is impossible.

\subsection{Definitions}\label{s:definitions}

In this section we formally introduce the SBM. 
For a rigorous definition of the processes we need to specify an appropriate 
state space. Let $\varphi_\lambda:\bbZ^d\mapsto\bbR, \varphi_\lambda(k)= e^{\lambda |k|}$, 
$\lambda\in \bbR$ and define
\begin{eqnarray}
 E_\textup{tem} &:=&\{\phi:\bbZ^d\mapsto\bbR_+:\langle 
\phi,\varphi_{\lambda}\rangle<\infty\text{ for all } \lambda<0\}\,,\\
 %E_\textup{rap}&:=&\{\phi:\bbZ^d\mapsto\bbR_+:\langle 
%\phi,\varphi_{\lambda}\rangle<\infty\text{ for all } \lambda>0\}\,,\\
 E_\textup{fin}&:=&\{\phi:\bbZ^d\mapsto\bbR_+: 
\sum_{k\in \bbZ^d} \phi(k) <\infty\}\,,\\
 E_\textup{cpt}&:=&\{\phi:\bbZ^d\mapsto\bbR_+:\phi(k)=0\text{ for 
all but 
finitely many }k\in\bbZ^d\}\,.
\end{eqnarray}
Following definitions on  page 1091 of \cite{DP98}, for any $\lambda\in \bbR$ and $u,v: \bbZ^d\mapsto\bbR$, set 
%\begin{equation}
$|u-v|_\lambda := \langle 
|u-v|,\varphi_\lambda\rangle$,
%\end{equation}
and thus we define metric on $ E_\textup{tem}$ as follows: 
\begin{equation}
d_{\text{tem}}(u,v):= \sum_{n=1}^\infty 2^{-n}
(|u-v|_{-\lambda_n}\wedge 1),
\end{equation}
where $\lambda_n\downarrow 0$.
% Similarly we
%topologize $E_\textup{rap}$ by choosing $\lambda_n\uparrow \infty$ and using the metric $
%d_{\text{rap}}(u,v):= \sum_{n=1}^\infty 2^{-n}
%(|u-v|_{\lambda_n}\wedge 1).$ 
Also $ E_\textup{fin}$ is topologized 
by the $l^1$-norm, $\|u-v|\|_1:= \sum_{k\in \bbZ^d} |u(k)-v(k)|$, and  $E_\textup{cpt}$   is topologized 
by the $l^\infty$-norm.

%\begin{equation}
% E_\lambda:=\{(\phi,\psi):\bbZ^d\to(\bbR_+)^2:\langle 
%\phi,\varphi_\lambda\rangle,\langle 
%\psi,\varphi_\lambda\rangle<\infty\}\,.
%\end{equation}
%We especially denote $E_\textup{fin}:=E_0$. In the following we think of 
%$\lambda<0$ as fixed once and for all even though obviously its precise value 
%has no significance. We also define
%\begin{equation}
% E_\textup{cpt}:=\{(\phi,\psi):\bbZ^d\to(\bbR_+)^2:\phi(k)+\psi(k)=0\text{ for 
%all but 
%finitely many }k\in\bbZ^d\}\,.
%\end{equation}
%We endow $E_\lambda$, $E_\textup{fin}$, and $E_\textup{cpt}$ with the topology 
%of componentwise convergence. Note that some authors refer to $E_\lambda$ as a 
%Liggett-Spitzer space. Here, we make a particular choice for the weight 
%function. 
% Let us also define
%\begin{equation}
% E_\textup{tem}:=\{(\phi,\psi):\bbZ^d\to(\bbR_+)^2:\langle 
%\phi,\varphi_{\lambda}\rangle,
 %\langle \psi,\varphi_{\lambda}\rangle<\infty\text{ for all } \lambda<0\}\,.
%\end{equation}
%The space is equipped with the topology defined on page 1091 of \cite{DP98}.

  %  In what follows with some abuse of notation for a function $\phi: \bbZ^d\to \bbR_+$ we say that 
%$\phi \in E_\lambda$ (resp. in $E_\textup{cpt}$, $E_\textup{fin}$, $E_\textup{tem}$) if 
%$(\phi,\phi) \in E_\lambda$ (resp. in $E_\textup{cpt}$, $E_\textup{fin}$, $E_\textup{tem}$) . 

Now $\Omega_{\text{tem}}$ (resp. % $\Omega_{\text{rap}}$, 
 $\Omega_{\text{fin}}$)
 is the space of $E_{\text{tem}}^2$-valued (resp.  %$E_{\text{rap}}^2$-valued, 
 $E_{\text{fin}}^2$-valued) continuous paths on $\bbR_+$ with the compact-open
topology. 

Whenever $x:\bbZ^d\mapsto\bbR_+$ and $x(i)=\theta\geq0$ for all $i\in\bbZ^d$ we use 
bold letters and write $x=\boldsymbol{\theta}$. We refer to this situation as a 
flat configuration.

In the following we denote the discrete Laplace operator by $\Delta$, that is, 
for $\phi:\bbZ^d\mapsto\bbR_+$ we set
\begin{equation}
 \Delta \phi(i):=\frac1{2d}\sum_{j\sim i} (\phi(j)-\phi(i))\,,\quad 
i\in\bbZ^d\,,
\end{equation}
where $j\sim i$ means $j$ is a neighbor of $i$. 

We denote the semigroup corresponding to the continuous-time rate $1$ simple symmetric 
random walk on $\bbZ^d$ by $(P_t)_{t\geq0}$ and the transition 
probabilities by $p_t(i,j)$, $t\geq0$, $i,j\in\bbZ^d$. That is, if 
$(Z_t)_{t\geq0}$ is a simple symmetric random walk on $\bbZ^d$ and 
$\phi:\bbZ^d\mapsto\bbR$ is a bounded function, then 
$P_t\phi(i)=E[\phi(Z_t)|Z_0=i]$ and 
$p_t(i,j)=P\{Z_t=j|Z_0=i\}$, where $t\geq0$ and $i,j\in\bbZ^d$. Also, 
we denote the Green's function of a simple symmetric random walk on $\bbZ^d$ ($d\geq 3$) 
 by $g(\cdot,\cdot)$, that is,
\begin{equation}
 g(x,y)=\int_0^\infty P_s\mathbbm{1}_{\{y\}}(x)\,\dx s=\int_0^t 
p_s(x,y)\,\dx s,\; x,y\in \bbZ^d. 
\end{equation}

Moreover if $X,Y$ are stochastic processes defined on the same probability 
space, we denote their quadratic co-variation by $[X,Y]_t$ and quadratic 
variation by $[X]_t$. Equality in distribution is denoted by $\eqd$.

%The space of continuous functions $\bbR_+\to E_\lambda$ equipped with the 
%topology of uniform convergence on compact sets is denoted by 
%$C(\bbR_+,E_\lambda)$. It will serve as state space for the SBM. 

\begin{definition}\label{d:solution}
Let $b>0$, $\rho=1$, and $(x,y)\in E_{\text{tem}}^2$. We say that a stochastic 
process $(u,v)=(u_t,v_t)_{t\geq0}$ is a Symbiotic Branching 
Model SBM$(1,b,x,y)$ on a filtered probability space 
$(\Omega,\mcA,\mcF=(\mcF_t)_{t\geq0},\bbP)$ if the sample 
paths of $(u,v)$ lie in $\Omega_{\text{tem}}$, and there exists a
family of planar Brownian motions 
$\{W(i)=(W_t(i))_{t\geq0}:i\in\bbZ^d\}$ adapted 
to the filtration $\mcF$ such that the following is satisfied:
 $(u,v)$ solves the following system of interacting stochastic 
differential equations 
\begin{equation}\label{SBM}
\left\{
\begin{array}{l}
 (u_0,v_0)=(x,y)\,,\\[1ex]
 \dx u_t(i)=\Delta u_t(i)\,\dx t+\sqrt{b u_t(i)v_t(i)}\,\dx W_t(i)\,,\quad
t\geq0, 
i\in\bbZ^d\,,\\[1ex]
 \dx v_t(i)=\Delta v_t(i)\,\dx t+\sqrt{b u_t(i)v_t(i)}\,\dx W_t(i)\,,\quad t\geq0, 
i\in\bbZ^d\,.
 \end{array}
 \right.
\end{equation}
 %\item The sample paths of the process $(u,v)$ lie in $C(\bbR_+, E_\lambda)$, 
%a.s.
\end{definition}
 In what follows, we assume throughout the paper that $$\rho=1.$$
 %, and thus we will use abbreviation SBM$(b,u_0,v_0)$ for this fully positively correlated symbiotic model. 
Also, let us note, that for simplicity, whenever confusion is impossible we omit some or all of the 
parameters in the abbreviation
SBM$(1,b,u_0,v_0)$.

%\begin{remark}
%The existence of the SBM is proved in Theorem 2.2 of \cite{DM13}. %{MR3111225}.
% Note 
%that they use a topology on $E_\lambda$ which is induced by the Liggett-Spitzer 
%norm (see Section 2.1.1 there). Since this topology is clearly finer than the 
%topology of componentwise convergence which we use, their result also 
%implies existence in our setting.
% Note if the topology is finer, then so is the sigma algebra. Hence if 
% they construct a distribution on the finer sigma algebra this also
% induces a dis
%\end{remark}

%For $\rho$ it can be easily shown that \eqref{SBM} has a unique strong 
%solution. We need this later since we want use a result for strong solutions.

Next proposition provides the existence and uniqueness result for~\eqref{SBM}. 
\begin{proposition}
Let $b>0$,  and $(x,y)\in E_{\text{tem}}^2$.% (resp. $ E_{\text{rap}}^2$). 
 The system of stochastic differential equations~\eqref{SBM} has a 
unique strong solution with sample paths in  $\Omega_{\text{tem}}$. % (resp.  $\Omega_{\text{rap}}$). 
\end{proposition}

\begin{proof}
Existence of a weak solution to~\eqref{SBM} is proved in Proposition~3.1 of~\cite{BDE11}. 

Now let us show the pathwise uniqueness. 
First note that $\eta_t:=v_t-u_t$ is deterministic and solves the 
discrete heat equation, that is $\dx \eta_t = \Delta \eta_t \,\dx t$.
Hence $(u,v)$ is uniquely determined if $u$ and $\eta$ are. For $\eta$ 
existence and uniqueness are known of course. For $u$ we have the following 
system of locally $1$-dimensional stochastic differential 
equations:
\begin{equation}\label{e:sde1}
 \dx u_t(i)=\Delta u_t(i)\,\dx t+\sqrt{b u_t(i)^2+bu_t(i)\eta_t(i)}\,\dx W^u_t(i)\,,\quad 
i\in\bbZ^d\,.
\end{equation}
For \eqref{e:sde1} pathwise uniqueness follows by the argument similar to the one used in the proof of  Theorem~3.2 in Shiga and 
Shimizu \cite{SS80}. Note that 
in~\cite{SS80} the diffusion coefficient does not depend on time. However, the Yamada-Watanabe 
argument used in \cite{SS80} obviously works just as well for our 
time-dependent diffusion coefficient since 
$\eta_t(i)%\leq \langle \eta_t,\boldsymbol{1}\rangle =\langle \eta_t,\boldsymbol{1}\rangle
<\infty$ for all $t\geq0$.
Also \cite{SS80} assumes that at each site the process takes values in $[0,1]$. It is easy to 
seen however that for the argument in~\cite{SS80}to work it is sufficient to have 
$\sup_{t\in[0,T]}\bbE[|u_t(i)-u'_t(i)|]<\infty$ for all $T\geq0$, for any two solutions $u, u'$ to ~\eqref{e:sde1}. This can 
be immediately seen as follows. Recall that $(P_t)_{t\geq0}$ denotes the 
semigroup corresponding to the simple symmetric random walk on $\bbZ^d$. Then, 
by Proposition 3.1 of \cite{BDE11}, $\bbE[u_t(i)]=P_t u_0(i)$ for every 
$i\in\bbZ^d$ and thus $\sup_{t\geq0}\bbE[u_t(i)]<\infty$.

By Yamada-Watanabe theorem weak existence and pathwise uniqueness imply that there exists unique strong solution to~\eqref{e:sde1}
(see e.g. Theorem~2.2 in~ \cite{SS80} for the analogous result). 
% assumption however is only 
% needed to ensure the existence of the series in 
% the Gronwall inequality. Since we only consider simple symmetric random walk for migration 
% we have a finite sum in this situation anyways, that is, we get 
% \begin{equation}
%  \bbE[|u_t(i)-u'_t(i)|]\leq \sum_j (\mathbbm{1}_{\{j\sim i\}}+\mathbbm{1}_{\{j= i\}}) \int_0^t 
% \bbE[|u_s(j)-u'_s(j)|] \,\dx s \,.
% \end{equation}
\end{proof}

In the study of the SBM an important role is played by the total mass process 
and many results can 
be deduced from the behavior of this simpler process. The total mass of an 
element 
$x:\bbZ^d\mapsto\bbR_+$, is denoted by 
\begin{equation}\label{e:totalmass}
 \bar{x}=\langle x,\boldsymbol{1}\rangle=\sum_i x(i)\,.
\end{equation}

\begin{proposition}[Proposition 3.2 of \cite{BDE11}]\label{l:totalmass}
If $u_0,v_0\in E_\textup{fin}$, then the total 
mass processes $\bar{u}=(\bar{u}_t)_{t\geq0}$ and 
$\bar{v}=(\bar{v}_t)_{t\geq0}$ are non-negative, 
continuous, square integrable martingales and
\begin{align}
 \dx \bar{u}_t&= \sqrt{b \langle u_t, v_t\rangle }\,\dx \tilde{W}_t\,, \quad t\geq0, \\
 \dx \bar{v}_t&= \sqrt{b \langle u_t, v_t\rangle }\,\dx \tilde{W}_t\,, \quad t\geq0, 
\end{align}
where $\tilde{W}$ is a  Brownian motion with 
the variance  $[\tilde{W}_\cdot,\tilde{W}_\cdot]_t=t$.
\end{proposition}

The lemma also implies that 
\begin{equation}
 [\bar{u},\bar{v}]_t=b\int_0^t \langle u_s,v_s\rangle\,\dx s\,, t\geq0. 
\end{equation}
Another consequence of the lemma is that by the martingale convergence theorem, 
the limits 
\begin{equation}
 \lim_{t\to\infty}\bar{u}_t=\bar{u}_\infty\quad\text{and}\quad\lim_{t\to\infty}
\bar{v}_t =\bar{v}_\infty\,,
\end{equation}
exist, almost surely.

What happens in the case of flat initial conditions, that is, for $(u_0,v_0)=
(\boldsymbol{\theta}_1,\boldsymbol{\theta}_2)$ with 
$\theta_1,\theta_2\geq0$? In this case, the existence of limit of $(u_t, v_t)$ in $E_{\text{tem}}^2$  is known for the case of $\rho\in (-1,1)$ (see Proposition 4.1 in \cite{BDE11}). 
As for the case of $\rho=1$, since expectations of $\langle u_t,\varphi_{\lambda}\rangle,  \langle u_t,\varphi_{\lambda}\rangle$ 
are constant in $t$ and thus bounded for any $\lambda>0$, we can immediately get that the family $\{(u_t, v_t)\,, t\geq 0\}$ is tight in 
$E_\textup{tem}^2$ (see Lemma~2.3(c) of~\cite{DP98} and its proof for analogous argument). Thus there exist weak  limit points $(u_\infty, v_\infty)\in E_\textup{tem}^2$. 

\begin{definition}[coexistence]
\begin{enumerate}
\item Assume that $(u_0,v_0)\in E_\textup{fin}^2$. 
We say that coexistence is 
possible if 
 \begin{equation}
  \bbP\{\bar{u}_\infty\bar{v}_\infty>0\}>0\,.
 \end{equation}
 We say that coexistence is impossible if
 \begin{equation}
  \bbP\{\bar{u}_\infty\bar{v}_\infty>0\}=0\,.
 \end{equation}
\item Assume that $(u_0,v_0)=
(\boldsymbol{\theta}_1,\boldsymbol{\theta}_2)$ with 
$\theta_1,\theta_2\geq0$.
   We 
say that coexistence is possible if for any  weak limit point $(u_\infty, v_\infty)$ of $\{(u_t, v_t)\,, t\geq 0\}$ there exists $(\phi,\psi)\in E_\textup{cpt}^2$ 
such that
 \begin{equation}
  \bbP\{\langle u_\infty,\phi\rangle\langle 
v_\infty,\psi\rangle >0\}>0\,.
 \end{equation}
 We say that coexistence is impossible if for any  weak limit point $(u_\infty, v_\infty)$ of $\{(u_t, v_t)\,, t\geq 0\}$  and for all $\phi,\psi\in E_\textup{cpt}$
 \begin{equation}
  \bbP\{\langle u_\infty,\phi\rangle\langle 
v_\infty,\psi\rangle >0\}=0\,.
 \end{equation}
\end{enumerate}
In the case of the PAM where the two populations are identical we use the 
notion of survival instead of coexistence.
\end{definition}

\subsection{Longtime behavior of 
PAM}\label{s:overview}

Before stating our main results for SBM with $\rho=1$, 
%we briefly review what is known about  the longtime behavior of SBM,
 we present  results on the survival/extinction  dichotomy for the PAM, that is, for the SBM$(1,b,u_0, v_0=u_0)$. 
In what follows,  we will use for this process the short notation PAM($b,u_0$) and drop the parameter $b$ 
or the initial condition whenever no confusion may occur. 

Recall that $g(\cdot,\cdot)$ 
denotes the Green's function of a simple random walk on $\bbZ^d$ and define
\begin{equation}
 b_2:=\frac2{g(0,0)}\,.
\end{equation}
Then the result on the longtime behavior for the PAM reads as follows.

\begin{theorem}[longtime behavior of PAM 
\cite{GDH07,BS10,BS11}]\label{t:DGH}
Let %$\rho=1$, 
$b>0$, and $(u_t)_{t\geq 0}$ be PAM($b,u_0$). 
\begin{enumerate}[(i)]
\item\label{i:flat} Assume that $u_0=\boldsymbol\theta>0$. 
Then 
%there  exists a probability measure $\nu_{\theta}$ on $E_\lambda$ such that 
%$\wlim_{t\to\infty}u_t=\nu_{\theta}$ with respect to $E_\lambda$.
\begin{enumerate}[(a)]
 \item if $d\in\{1,2\}$, then %$\nu_{\theta}=\delta_\mathbf{0}$, that is 
survival is impossible.
 \item if $d\geq3$, then there exists $b_\ast> b_2$ such that 
 %\begin{equation}
 % \nu_\theta\begin{cases}
  %    =\delta_\mathbf{0}\,,&\text{if }b>b_\ast\,,\\
    %  \neq\delta_\mathbf{0}\,,&\text{if }b<b_\ast\,.
    % \end{cases}
 %\end{equation}
 %That is, 
for $b<b_\ast$, survival is possible and for $b>b_\ast$ survival is 
impossible.
 \end{enumerate}
 \item Assume that $u_0\in E_\textup{cpt}$. Then survival is possible if $d\geq 3$ and $b<b_\ast$, and survival is impossible if $d=1,2$ or $d\geq 3$ and $b>b_\ast$.  
\end{enumerate}
\end{theorem}

\begin{remark}
In fact Greven and den Hollander \cite{GDH07} prove only $b_2\leq b_\ast$. They 
conjectured that in fact  $b_2<b_\ast$. The fact that   $b_2<b_\ast$
follows from results of Birkner and Sun \cite{BS10} for dimensions $d\geq4$
and \cite{BS11} for dimension $d=3$. \cite{GDH07} 
also show that the second moments of the PAM at each site 
$i\in\bbZ^d$ are bounded in time if and only if $b<b_2$. 
%Hence the fact that 
%$b_2<b_\ast$ also confirms the existence of an intermediate regime 
%where the PAM has a nontrivial equilibrium distribution which (locally) has no 
%second moments.
%For for details about this we refer to Proposition \ref{p:BirkSun}.
\end{remark}

\begin{proof}[Proof of Theorem \ref{t:DGH}]
 For flat initial conditions, the result is contained in Theorems 1.2, 1.3, 1.4,
and 1.5 of \cite{GDH07}. For summable initial conditions the longtime behavior 
can be carried over from the flat setting using the self-duality from Lemma~\ref{l:selfdual} as follows: for all $\theta>0$ and $u_0\in E_{\rm cpt}$ 
\begin{align}\label{e:selfdual}
 \bbE_{u_0}[e^{-\theta \bar{u}_\infty}]&=
 \lim_{t\to\infty}\bbE_{u_0}[e^{-\theta \langle \mathbf{1},u_t\rangle}]
 =\lim_{t\to\infty}\bbE_{\boldsymbol\theta}[ e^{-\tilde u_t ,u_0\rangle}]\,,
\end{align}
where $(\tilde u_t)_{t\geq 0}$ is  PAM($b,{\boldsymbol\theta}$). 
By Theorem 1.3 and 1.4 of \cite{GDH07}, the limiting law of $\tilde u_t$ exists, it is 
translation invariant and this implies that the right hand side of 
\eqref{e:selfdual} equals $1$ if and only if $\wlim_{t\to\infty} \tilde u_t=\mathbf{0}$. Here and elsewhere $\wlim$ denotes weak (in probability sense) limit. 
\end{proof}

%In this paper we are interested in fully positively correlated case of $\rho=1$. 

\subsection{Our Main Result}\label{s:mainresult}

%From the above overview of theorems from the literature we 
%see that in the understanding of the longtime behavior of SBM there is a gap 
%for $\rho\in(0,1]$. In this paper we will treat the boundary case of $\rho=1$. % in the transient case.

%We start with the transient case.

 Recall the parameter $b_\ast$ from 
Theorem \ref{t:DGH}. We establish  the longtime behavior of the SBM for  the 
case  of  completely positive correlations ($\rho=1$). % $u_0\not=v_0$. 
Recall that, as we mentioned above, if it is not stated otherwise, we assume that $\rho=1$ in what follows. 
\begin{theorem}[Longtime behavior for  summable initial conditions]\label{t:main}
 Let %$\rho=1$ and 
  $(u_0,v_0)\in E_\textup{fin}^2$ such 
that $\bar{u}_0\bar{v}_0>0$. %Let, for $d\geq 3$,  $b_\ast$ be as in Theorem \ref{t:DGH}.
\begin{enumerate}[(i)]
%\item\label{i:recurrent} 
 \item\label{i:sumcoex} 
%\begin{enumerate}
%\item\label{i:coex}
Let $d\geq3$. % and   $b_\ast$ be as in Theorem \ref{t:DGH}. 
Then 
for all $b\in (0, b_\ast)$
coexistence for SBM($1,b,u_0,v_0$) is 
possible. 
%\item\label{i:moments} For every $\gamma\in(0,1]$ there exists 
%$b_{1+\gamma}\in[b_2,b_\ast)$ such that for 
%\begin{equation}
%b\in\biggl(0, 
%\frac{2b_{1+\gamma}}{1+\rho}\biggr) 
%\end{equation}
%and $(u,v)$ an SBM($\rho,b,u_0,v_0$),
%\begin{equation}
% \sup_t \bbE_{(u_0,v_0)}[(\langle u_t,\boldsymbol{1}\rangle+\langle 
%v_t,\boldsymbol{1}\rangle)^{1+\gamma}]<\infty
%\end{equation}
%\end{enumerate}
 \item\label{i:sumnoncoex} 
Let 
\begin{itemize}
\item[(a)] $d\in\{1,2\}$ and $b>0$, 
\end{itemize}
or
\begin{itemize}
\item[(b)]
  $d\geq3$ and  $b>b_\ast$. 
\end{itemize}
%For any $b>0$ in the case of $d\in\{1,2\}$  and for any $b>b_\ast$ in the case of $d\geq 3$,
Then coexistence  for SBM($1,b,u_0,v_0$)  
is impossible. Moreover, in both cases (a) and (b), if $\bar{u}_0\leq \bar{v}_0$, then 
\begin{equation}
 \bar u_t \to 0\,.
\end{equation}
and 
\begin{equation}
 \bar v_t \to  \bar v_0- \bar u_0\,.
\end{equation}
almost surely, as $t\to\infty$. 
\end{enumerate}
\end{theorem}

\begin{theorem}[Longtime behavior for flat initial conditions]\label{t:main2}
	% Let $\rho=1$.
	Assume that  $(u_0,v_0)=
	(\boldsymbol{\theta}_1,\boldsymbol{\theta}_2)$ with 
	$\theta_1, \theta_2>0$. 
%Let  $b_\ast$ be as in Theorem \ref{t:DGH}. Then 
	\begin{enumerate}[(i)]
		\item\label{i:sumcoex_1} 
		%\begin{enumerate}
		%\item\label{i:coex}
Let $d\geq3$. % and   $b_\ast$ be as in Theorem \ref{t:DGH}. 
Then 
		for all $b\in (0,b_\ast)$
		coexistence for SBM($1,b,u_0,v_0$) is 
		possible. 
		%\item\label{i:moments} For every $\gamma\in(0,1]$ there exists 
		%$b_{1+\gamma}\in[b_2,b_\ast)$ such that for 
		%\begin{equation}
		%b\in\biggl(0, 
		%\frac{2b_{1+\gamma}}{1+\rho}\biggr) 
		%\end{equation}
		%and $(u,v)$ an SBM($\rho,b,u_0,v_0$),
		%\begin{equation}
		% \sup_t \bbE_{(u_0,v_0)}[(\langle u_t,\boldsymbol{1}\rangle+\langle 
		%v_t,\boldsymbol{1}\rangle)^{1+\gamma}]<\infty
		%\end{equation}
		%\end{enumerate}
		\item\label{i:sumnoncoex}
%For any $b>0$ in the case of $d\in\{1,2\}$  and for any $b>b_\ast$ in the case of $d\geq 3$,
Let 
\begin{itemize}
\item[(a)] $d\in\{1,2\}$ and $b>0$, 
\end{itemize}
or
\begin{itemize}
\item[(b)]
  $d\geq3$ and  $b>b_\ast$. 
\end{itemize}
Then coexistence  for SBM($1,b,u_0,v_0$)  is impossible.
  Moreover, in both cases (a) and (b), if $\theta_1\leq \theta_2$, then for any $\phi \in E_\textup{cpt}$% (or maybe in $E_\textup{fin}$????)
		\begin{equation}
		\langle u_t, \phi\rangle\to 0\,,
		\end{equation}
and 
\begin{equation}
		\langle v_t, \phi\rangle\to (\theta_2-\theta_1)\langle 1, \phi\rangle \,,
		\end{equation}
		in probability as $t\to\infty$. 
	\end{enumerate}
\end{theorem}

Let us comment on the above results. First of all, as we see, the transition threshold between possible coexistence and non-coexistence for the SBM($1,b, u_0, v_0$) is exactly the same as for the PAM model which is, in fact,  SBM($1,b, u_0, v_0=u_0$) model. This may sound like as a pretty much  well expected result, however we found that the proof of it is surprisingly not very  straightforward. It does use close connection between 
SBM($1$) and PAM models, however in the regime of non-coexistence with summable {\it non-monotone} initial conditions (that is,  $u_0\not \leq v_0$ and  $v_0\not\leq u_0$)
on top of comparison with the PAM, one uses non-trivial decomposition of the SBM  and some interesting PDE results (see Section~\ref{s:noncoexsum} for this argument). We hope that our proofs  with give an additional motivation for studying the open question of existence of the phase transition in the transient regime for 
$\rho\in (0,1)$. %, where  we conjecture that the coexistence/non-coexistence threshold is at the level $\frac{b_\star}{\rho}$. 

%\begin{figure}[t]
% \centering
%\begin{tikzpicture}[scale=1] 
%\draw[->] (-0.3,0) -- (7,0);
%\draw (8,0) node[below] {$\gamma$};
%\draw[->] (0,-0.3) -- (0,5) node[left] {$b_{1+\gamma}$};
%\draw (2,.1) -- (2,-.1) node[below] {$p(\rho)-1$};
%\draw[dashed] (2,0) -- (2,1.8);
%\draw[dashed] (0,2) -- (2,2);
%\draw (-.1,2) -- (.1,2) node[left] {$b_\#(\rho)\;$};
%\draw (6,.1) -- (6,-.1) node[below] {$1$};
%\draw (-0.1,1) -- (0.1,1) node[left] {$b_2\;$};
%\draw (0,3.5) .. controls (2,1.8) and (3,1.2) .. (6,1);
%\draw[dashed] (0,1) -- (6,1);
%\fill[color=white] (0,3.5) circle (3pt);
%\draw[color=black] (0,3.5) circle (3pt);
%\fill[color=black] (0,4.3) circle (3pt) node[left] {$b_\ast$};
%\end{tikzpicture} 
%\caption{
% Qualitative picture of the critical parameter curve. (It remains to be proved 
%that is montone and continuous.)
%}
%\end{figure}

\section{Proof of Theorems \ref{t:main}(i) and \ref{t:main2}(i): Coexistence 
Possible}\label{s:proofcoex}

\subsection{Preparations} 

In Section \ref{s:proofcoex} we prove the coexistence parts of Theorems~ 
\ref{t:main}, \ref{t:main2}. The actual proof will be carried out in Subsection 
\ref{s:proofcoexistfin}. In this subsection we will review and prove a couple 
of results for the PAM which we will need for the proof of Theorems 
\ref{t:main}(i), \ref{t:main2}(i). 

%Recall that SBM($1,b,u_0,v_0=u_0$) coincides with a Parabolic Anderson 
%Model.
Let $w$ be PAM($b, w_0$), that is, $w$  satisfies the following equation:
\begin{equation}
\dx w_t(i)=\Delta w_t(i)\,\dx t+\sqrt{b w^2_t(i)}\,\dx W_t(i), \quad t\geq0,  i\in \bbZ^d. 
\end{equation}
The PAM exhibits the following simple but very usefull self-duality property, see 
Section 2 of Cox, Klenke and Perkins \cite{CKP00}: let $w$ and $\tilde{w}$ be 
PAM$(b)$ processes such that $w_0=\boldsymbol{\theta}$ and 
$\tilde{w}_0=\phi\in E_\textup{fin}$. Then
\begin{align}
  \bbE_{\tilde{w}_0}[e^{-\lambda\langle \tilde{w}_t,\boldsymbol{\theta}\rangle}]
  =\bbE_{\tilde{w}_0}[e^{-\langle 
\tilde{w}_t,\lambda\boldsymbol{\theta}\rangle}]
 =\bbE_{\lambda\boldsymbol\theta}[e^{-\langle \tilde{w}_0,u_t\rangle}]
  =\bbE_{\boldsymbol\theta}[e^{-\langle \tilde{w}_0,\lambda w_t\rangle}]
  =\bbE_{\boldsymbol\theta}[e^{-\lambda\langle \phi,w_t\rangle}]\,, \quad \forall t\geq0. 
 \end{align}
%The self-duality is due to the fact that the 
%PAM is a linear system, that is, if $w^1$ and $w^2$ are PAM($b,w^1_0$) and 
%PAM($b,w^2_0$), espectively, then $w^1+w^2$ is a PAM($b,w_0^1+w_0^2$). It is 
%useful for example to relate processes started in finite and flat initial 
%conditions. 
From the self-duality we immediately obtain the following statement.

\begin{lemma}[self-duality]\label{l:selfdual}
Let $w$ and $\tilde{w}$ be PAM(b) such that $w_0=\boldsymbol{\theta}$ and 
$\tilde{w}_0=\phi\in E_\textup{fin}$. 
Then,
\begin{equation}
 \langle \tilde{w}_t,\boldsymbol{\theta}\rangle\eqd\langle \phi,w_t\rangle\,,\;\forall t\geq 0. 
\end{equation} 
\end{lemma}

%We next recall a result on the moments of the invariant distribution of 
%the PAM given in Theorem~\ref{t:DGH}.

%Proposition \ref{p:BirkSun} proves Conjecture 1.8 from \cite{GDH07}. Note that 
%in the proof we basically just clarify the connection between this conjecture 
%and Theorem 1.3 of Birkner and Sun \cite{BS10}.

We now prove a result which allows us to bound the moments of appropriate functions of SBM by those of 
PAM. Since we know a lot of  information about  PAM this will be very useful.

\begin{proposition}[comparison]\label{p:momentcomp}
 Let $(u,v)$ be the solution of 
SBM$(1,b,u_0,v_0)$. Let $w$ be a PAM$(b,w_0)$ such that 
$w_0=u_0+v_0$. Let $\Phi(t)$  be arbitrary non-negative  non-decreasing convex function on $\bbR_+$.  
\begin{itemize}
\item[(i)] 
If  
 $(u_0,v_0)\in E_\textup{fin}$, then 
\begin{equation}\label{e:moment_summable1}
 \bbE_{(u_0,v_0)}[\Phi(\langle u_t,\boldsymbol{1}\rangle+\langle 
v_t,\boldsymbol{1}\rangle)]
 \leq\bbE_{w_0}[\Phi(\langle w_t,\boldsymbol{1}\rangle)], \quad \forall t\geq0.  
\end{equation}
\item[(ii)]
If $(u_0,v_0)=(\boldsymbol{\theta}_1,\boldsymbol{\theta}_2)$ for 
$\theta_1,\theta_2\geq0$, then for all $\phi\in E_\textup{fin}$ 
such 
that $\phi\geq0$ one has
\begin{equation}\label{e:momentbflat}
 \bbE_{(u_0,v_0)}[\Phi(\langle u_t,\phi\rangle+\langle v_t,\phi\rangle)]
 \leq\bbE_{w_0}[\Phi(\langle w_t,\phi\rangle)],  \quad \forall t\geq0. 
\end{equation}

\end{itemize}
\end{proposition}

\begin{proof}
	We start with proving~(i). The idea is to use comparison result  %Theorem 2 
of Greven, Klenke and Wakolbinger  \cite{GKW02} (the earlier version of this result for the homogeneous case is due to 
Cox, Fleischmann, and Greven \cite{CFG96}); since the conditions of the result in~ \cite{GKW02} on function $\Phi$  are not satisfied in the proposition we will use the usual technique of approximation.  Let $\Lambda^N=[-N,N]^d$ be a torus, and $\Delta$ acting on 
 functions on the torus will be Laplacian with periodic boundary conditions.  
Let  $(u^N,v^N)$ be a solution of 
\begin{equation}\label{SBMN}
\left\{
\begin{array}{l}
% (u^N_0,v^N_0)\in [0,N]^{\Lambda_N}\times [0,N]^{\Lambda_N}\,,\\[1ex]
 \dx 
 u^N_t(i)=\Delta u^N_t(i)\,\dx t+\sqrt{b u^N_t(i)v^N_t(i)\left(1-\frac{u^N_t(i)}{N}\right)\left(1-\frac{v^N_t(i)}{N}\right)}\,\dx W_t(i)\,,\quad 
 t\geq0, i\in \Lambda_N\,,\\[1ex]
 \dx v^N_t(i)=\Delta v^N_t(i)\,\dx t+\sqrt{b u^N_t(i)v^N_t(i)\left(1-\frac{u^N_t(i)}{N}\right)\left(1-\frac{v^N_t(i)}{N}\right)}\,\dx W_t(i)\,,\quad 
 t\geq0, i\in \Lambda_N\,.
 \end{array}
 \right.
\end{equation}
with $u^N_0(i)=u_0(i), v^N_0(i)=v_0(i), i\in \Lambda_N\,.$
The solutions of this equation take values in $ [0,N]^{\Lambda_N}\times [0,N]^{\Lambda_N}$. 

%(\ref{SBM}) on the torus $\Lambda^N=[-N,N]^d$

%Define 

%In the following it is understood that $i\in\bbZ^d$ and $t\geq0$.
In the following it is understood that $i\in\Lambda_N$ and $t\geq0$.
Let 
\begin{equation}
\eta^N_t(i):= v^N_t(i)-u^N_t(i)
\end{equation}
and
\begin{equation}
\xi^N_t(i):= u^N_t(i)+v^N_t(i)= 2u^N_t(i)+\eta^N_t(i)\;\;\text{so that} 
\;\;u^N_t(i)=\frac12(\xi^N_t(i)-\eta^N_t(i))\,.
\end{equation}
Then
\begin{align}
  u^N_t(i)v^N_t(i) 
% & = u_t(i)(u_t(i)+\eta_t(i)) = u_t(i)^2+ u_t(i)\eta_t(i)\\
%  & = \frac14 ( \xi_t(i)-\eta_t(i) )^2 + \frac12 ( \xi_t(i)-\eta_t(i) )\\
%  &= \frac14 \left( (\xi_t(i)-\eta_t(i))^2+2  (\xi_t(i)-\eta_t(i))\right) \\
 = \frac14 \left( \xi^N_t(i)^2-\eta^N_t(i)^2\right) \,.
\end{align}
%The quadratic variation of $(\xi_t(k))_{t\geq0}$ is given by 
%\begin{align}
% [\xi^N(i)]_t&=[u^N(i)]_t+ 2[u^N(i),v^N(i)]_t+[v^N(i)]_t\\
% &=\int_0^t b \, u_s(i)v_s(i)\,\dx s+
% 2\int_0^t \rho\, b\, u_s(i)v_s(i)\,\dx s+
 %\int_0^t b \,u_s(i)v_s(i)\,\dx s\\
% &= \int_0^t 2(1+\rho) b\,u_s(i)v_s(i)\,\dx s\,.
% &=\int_0^t b \left( 
%\xi_s(i)^2-\eta_s(i)^2\right)\left(1-\frac{\xi^N_t(i)-\eta^N_t(i)}{N}\right)\left(1-\frac{\xi^N_t(i)+\eta^N_t(i)}{N}\right) \,\dx s\,.
%\end{align}
Therefore, 
%there exists an independent family of Brownian motions 
%$\{W^\xi(i):i\in\bbZ^d\}$ such that
we get that $\xi^N$ satisfies the following system of equations:
\begin{align}
 \dx \xi^N_t (i)
%  & =  \Delta \xi_t(i) \,\dx t + 2\sqrt{2(1+\rho)b u_t(i)v_t(i)}\,\dx W_t(i)\\
%  & =  \Delta \xi_t(i) \,\dx t + 2\sqrt{ \frac{2(1+\rho)b}4 \left( 
% \xi_t(i)^2-\eta_t(i)^2\right) 
% }\,\dx W_t(i)\\
 & =  \Delta \xi^N_t(i) \,\dx t + \sqrt{b ( 
\xi^N_t(i)^2-\eta^N_t(i)^2)\left(1-\frac{\xi^N_t(i)-\eta^N_t(i)}{2N}\right)\left(1-\frac{\xi^N_t(i)+\eta^N_t(i)}{2N}\right) }\,\dx W_t(i)\,,\;\; t\geq0, i\in \Lambda_N\,.
\end{align}
Clearly since the noises for $u^N$ and $v^N$ are  the same we get that $\eta^N$ is deterministic that solves the following system of equations:
\begin{align}
\dx \eta^N_t (i)
%  & =  \Delta \xi_t(i) \,\dx t + 2\sqrt{2(1+\rho)b u_t(i)v_t(i)}\,\dx W_t(i)\\
%  & =  \Delta \xi_t(i) \,\dx t + 2\sqrt{ \frac{2(1+\rho)b}4 \left( 
% \xi_t(i)^2-\eta_t(i)^2\right) 
% }\,\dx W_t(i)\\
& =  \Delta \eta^N_t(i) \,\dx t\,,\;\; t\geq0, i\in \Lambda_N\,.
\end{align}
Let now $w^N=(w^N_t)_{t\geq0}$ be an approximate PAM$(b,w_0)$. That is, for an 
independent family of Brownian motions $\{W^w(i):i\in\bbZ^d\}$, let $w^N$ on $\Lambda_N$ satisfy
\begin{align}
 \dx w^N_t(i)=  \Delta w^N_t(i) \,\dx t +
  \sqrt{b w^N_t(i)^2\left(1-\frac{w^N_t(i)}{2N}\right)}\,\dx W^w
_t(i)\,,\quad  t\geq0,  i\in \Lambda_N\,.
\end{align}
with $w^N_0(i)=w_0(i), i\in \Lambda_N\,.$
Let
\begin{align}
g_N^{\xi}(i,x,t)& := b (x^2 - \eta^N_t(i)^2)\left(1-\frac{x-\eta^N_t(i)}{2N}\right)\left(1-\frac{x+\eta^N_t(i)}{2N}\right)\,,\quad
x\,, t\geq0\,, i\in \Lambda_N\,,
\\
 g^w_N(i,x)&:= b
x^2\left(1-\frac{x}{2N}\right)\,,\quad x\,, t\geq0\,, i\in \Lambda_N\,.
\end{align}
Then, for all $x,t\geq0, i\in\Lambda_N$
\begin{equation}\label{e:compdifffct}
g_N^{\xi}(i,x,t) \leq  g_N^w(i,x)\,.
\end{equation}
The respective generators of $\xi^N$ and $w^N$ are given by the closure of the 
following operators
\begin{align}
 G_{N,t}^\xi f(\phi)&= \sum_{i\in \Lambda_N} \Delta \phi(i) \,\partial_i f(\phi) +\frac12  \sum_{i\in \Lambda_N}  
g_N^\xi(\phi(i),t)\,
\partial_{i} ^2 f(\phi)\,,\quad \phi\in(\bbR_+)^{ \Lambda_N}\,,\\
G^w_N f(\phi)&= \sum_{i\in \Lambda_N} \Delta \phi(i) \,\partial_i f(\phi) +\frac12  \sum_i 
g_N^w(\phi(i))\,
\partial_i^2 f(\phi)\,,\quad \phi\in(\bbR_+)^{ \Lambda_N}\,,
\end{align}
acting on functions $f:(\bbR_+)^{ \Lambda_N}\mapsto\bbR$ %which depend on finitely many coordinates only  and 
such that all their partial derivatives up to order $2$ exist and are continuous. % and bounded on the compacts. 
Now let $\mcD_N$ 
consist of functions $f$ as specified in the previous sentence with the additional requirement that 
$\partial_i\partial_j f\geq0$ for all $i,j\in \Lambda_N$.
Let $C^{2}_{bc}$ be the set of  convex functions 
 $F:\bbR_+\mapsto\bbR$ with bounded on the compacts continuous partial derivatives 
of orders $m=0,1,2$. Fix  arbitrary $F\in C^{2}_{bc}$. 
For $B_r$ the open ball of radius $r$ in $\bbZ^d$, %$and $p\geq 1$, 
let 
$F_{r}:(\bbR_+)^{\bbZ^d}\mapsto\bbR_+$, be defined by 
\begin{equation}
 F_{r}(x)=F\biggl(\,\sum_{i\in B_r} x(i)\biggr)\,.
\end{equation}
Then, for all $r\in (0,N)$,  we get that $F_{r}\in\mcD_N$,
% Moreover, if $(S_{s,t})_{t\geq s\geq 0}$ denotes the evolution 
%system generated by $(G_t^\xi)_{t\geq0}$ and $(T_t)_{t\geq0}$ denotes the semigroup generated by 
%$G^w$, then Lemma 15 of \cite{CFG96} ensures that $T_t \mcD$ is contained in 
%the domain of $G^w$ 
%and 
%Lemma 15 plus the results from Section 2.2 of \cite{CFG96} imply that $S_{s,t}$ preserves 
%$\mcD$. Note that the proof of Proposition 16 of \cite{CFG96} obviously goes through without 
%difficulties in our time-inhomogeneous setting. Hence,
and taking into account 
\eqref{e:compdifffct} we can apply %the comparison result
 Theorem 2 of 
\cite{GKW02} to obtain that for each $r\in (0,N)$,
\begin{equation}
\label{eq:29_08_01}
 \bbE_{u_0^N+v_0^N}[F_{r}(\xi^N_t)]
 \leq \bbE_{u_0^N+v_0^N}[F_{r}(w^N_t)]\,,\; \forall t\geq 0. 
\end{equation}
It is easy to check that as $N\rightarrow\infty$, $(\xi^N,\eta^N)$ converges weakly  (in the space of continuous $E^2_{\text{tem}}$-valued paths)  to $(\xi,\eta)$,
where $\xi_{\cdot}=u_{\cdot}+v_{\cdot},$ \mbox{$ \eta_{\cdot}=v_{\cdot}-u_{\cdot}$} and $(u,v)$ is a solution to the SBM($(1,b, u_0, v_0$). Moreover $w^N$ converges weakly (in the space of continuous $E_{\text{tem}}$-valued paths)  to $w$ which solves PAM$(b, w_0=u_0+v_0)$.  
Thus, letting $N\to\infty$, passing to the limit in~\eqref{eq:29_08_01} and using continuity of $F_r$ on $E_{\text{tem}}$, we get 
\begin{equation}
 \bbE_{u_0+v_0}[F_{r}(\xi_t)]
 \leq \bbE_{u_0+v_0}[F_{r}(w_t)]\,, \forall t\geq 0.
\end{equation}
Taking the limit $r\to\infty$ on both sides and using the monotone 
convergence theorem yields
\begin{equation}
 \bbE_{(u_0,v_0)} [F(\langle u_t,\boldsymbol{1}\rangle+\langle 
v_t,\boldsymbol{1}\rangle)]
\leq 
\bbE_{u_0+v_0}[F(\langle w_t,\boldsymbol{1}\rangle)]\,,  \forall t\geq 0.
\end{equation}
Then by approximating  a non-negative  non-decreasing convex function $\Phi$ by functions from $C^{2}_{bc}$ we can easily finish the proof of~\eqref{e:moment_summable1} by passing to the limit.

The proof of~(ii) goes along the same lines with 
\begin{equation}
 F_{r}(x)=F\biggl(\,\sum_{i\in B_r} x(i)\phi(i)\biggr)
\end{equation}
and will thus be omitted.
\end{proof}

\begin{corollary}[Uniform integrability]
\label{cor:UI_SBM}
 Let $d\geq3$. % and $\rho=1$.  
 Let  $b_\ast$ be as in Theorem \ref{t:DGH} and let 
 $b\in (0,b_\ast)$.  Let $(u,v)$ be the solution of 
SBM$(1,b,u_0,v_0)$.
\begin{itemize}
\item[(a)] 
If $(u_0,v_0)\in E_\textup{fin}^2$,  then $\{\bar u_t\,, t\geq 0\}$ and  $\{\bar v_t\,, t\geq 0\}$ are uniformly integrable. 
\item[(b)] 
If $(u_0,v_0)=(\boldsymbol{\theta}_1,\boldsymbol{\theta}_2)$ for 
$\theta_1,\theta_2\geq0$, then for all $\phi\in E_\textup{fin}$ 
%such  that $\phi\geq0$, 

  $$\{ \langle u_t\,, \phi\rangle\,, t\geq 0\}  \;\text{and}\;  \{ \langle v_t\,, \phi\rangle \,, t\geq 0\} \text{ are uniformly integrable.}$$ 
%\begin{enumerate}
%\item\label{i:coex}

\end{itemize}

\end{corollary}
\begin{proof}
Let $w$ be a PAM$(b,w_0)$ such that 
$w_0=u_0+v_0$. 

Let us first prove~(b), that is, the case of flat initial conditions.  
In this case the uniform integrability of $\{\langle w_t,\phi\rangle, t\geq 0\}$, for  $\phi\in E_\textup{fin}$,  follows easily from Theorems  1.3, 1.4,
 of~\cite{GDH07}.  Then by necessary and sufficient criterion for uniform integrability, for any $\phi\in E_\textup{fin}$,  we get the existence of non-negative non-decreasing convex function 
$\Phi(t)$ on  $\bbR_+$ such that $\lim_{t\to\infty}\frac{\Phi(t)}{t}<\infty$ and 
$$ \sup_{t\geq 0} \bbE\left[ \Phi (\langle w_t,\phi\rangle)\rangle\right] =\infty. $$ 
Then Proposition~\ref{p:momentcomp}(ii) and again the criterion for uniform integrability imply that both $\{\langle u_t,\phi\rangle, t\geq 0\}$ and $\{ \langle v_t,\phi\rangle, t\geq 0\}$ are 
unifromly integarble for any $\phi\in E_\textup{fin}$. 

For the proof of~(a) we use the self-duality of PAM proved in Lemma~\ref{l:selfdual} and again  Theorems  1.3, 1.4,
 of~\cite{GDH07} to get uniform integrability of $\{\bar w_t\,, t\geq 0\}$ and then we use  Proposition~\ref{p:momentcomp}(i) and follow the lines of the proof for the case (b).  
\end{proof}

\subsection{Proof of Theorem \ref{t:main}(\ref{i:sumcoex})}\label{s:proofcoexistfin}

With the preparations from the previous section we are now in a position to 
prove Theorem \ref{t:main}(\ref{i:sumcoex}), the coexistence result for 
summable initial conditions.

We need to show that 
 \begin{equation}
\label{eq:29_07_2}
  \bbP\{\bar{u}_\infty\bar{v}_\infty>0\}>0\,.
 \end{equation}

Without loss of generality  assume that $\bar u_0 \leq \bar v_0$. Note that $\bar\eta_t\equiv\bar v_t- \bar u_t=\bar v_0- \bar u_0\geq 0$
for all $t$ since $\eta$ solves the heat equation  and  its total mass remains constant. Therefore, by sending $t$ to infinity we get 
\begin{equation}
\label{eq:29_07_1}
\bar v_\infty \geq \bar u_\infty\,, \; \bbP-{\rm a.s.}
\end{equation}
By Corollary~\ref{cor:UI_SBM}(a) $\{\bar u_t\,, t\geq 0\}$ is uniformly integrable, and thus we get 

\begin{align}
\bbE_{u_0}[\bar u_\infty]=
 \bbE_{u_0}[\lim_{t\to\infty}\bar u_t]
 =\lim_{t\to\infty} \bbE_{u_0}[\bar u_t]=\bar u_0>0\,.
\end{align}
Thus $  \bbP\{\bar{u}_\infty>0\}>0$ and  using~\eqref{eq:29_07_1} we get~\eqref{eq:29_07_2}. 

\qed

\subsection{Proof of Theorem \ref{t:main2}(\ref{i:sumcoex_1})}
The proof of the  coexistence result for 
flat initial conditions goes along the similar lines as the proof for summable initial conditions. 
Let  $(u_\infty, v_\infty)$ be an arbitrary weak limit point  of $\{(u_t, v_t)\,, t\geq 0\}$.    
We will show a little bit more than required  in the definition of coexistence. Let  
  $(\phi,\phi)$ be an arbitrary element in $E_\textup{cpt}^2$. 
  We will show that 
 \begin{equation}
 \label{eq:1_08_2}
  \bbP\{\langle u_\infty,\phi\rangle\langle 
v_\infty,\phi\rangle >0\}>0\,.
 \end{equation}

Without loss of generality  assume that $\theta_1 \leq  \theta_2$. 
Note that $\eta_t(x)\equiv v_t(x)- u_t(x)\equiv \theta_2- \theta_1\geq 0$
for all $t\geq 0$ and $x\in \bbZ^d$ since $\eta$ solves the heat equation  starting at constant initial condition $\boldsymbol{\theta}_2-\boldsymbol{\theta}_1$.
% and that its total mass remains constant.
 Therefore, by sending $t$ to infinity we get 
that for  an arbitrary weak limit point $(u_\infty, v_\infty)$ the following holds:
\begin{equation}
\label{eq:1_08_1}
 v_\infty (x) \geq  u_\infty(x)\,, \forall x\in \bbZ^d\,,\; \bbP-{\rm a.s.}
\end{equation}
By Corollary~\ref{cor:UI_SBM}(b) $\{\langle u_t,\phi\rangle, t\geq 0\}$ is uniformly integrable, and thus we get 

\begin{align}
\bbE[\langle u_\infty,\phi\rangle]
 =\lim_{t\to\infty} \bbE_{(\boldsymbol{\theta}_1, \boldsymbol{\theta}_2)}[\langle u_t,\phi\rangle]=\theta_1\langle 1,\phi\rangle>0\,.
\end{align}
Thus $  \bbP\{\langle u_\infty,\phi\rangle>0\}>0$ and  using~\eqref{eq:1_08_1} we get~\eqref{eq:1_08_2}. 

\qed

%Hence  $v$ survives with probability 1. 

%So assume that $\phi\in E_\textup{fin}$ and $\phi\geq0$. Then by Proposition 3.1 
%from \cite{BDE11} 
%we know that (here $P_t$ is the HE semigroup),

%The non-coexistence part of Theorem \ref{t:main} will be proved in a separate 0
%Section \ref{s:noncoexsum}.

\section{Proof of 
Theorems\ref{t:main}(\ref{i:sumnoncoex}) and~\ref{t:main2}(\ref{i:sumnoncoex}): Coexistence 
Impossible.}
\label{s:proofnoncoex}

\subsection{Proof of Theorem \ref{t:main2}(\ref{i:sumnoncoex})}\label{s:noncoexflat}
We are now turning our attention to the proof of
Theorem\ref{t:main2}(\ref{i:sumnoncoex}) which states that coexistence is 
impossible for flat initial conditions in the recurrent case ($d\leq 2$) for arbitrary $b>0$ and in the transient case ($d\geq 3$) for arbitrary $b>b_\ast$. 
%We start with the case of flat initial conditions which  is much simpler. 

If in addition, we assume $u_0=v_0=\boldsymbol{\theta}>0$, then the SBM coincides with the PAM, whose longtime 
behavior has been studied extensively in \cite{GDH07}, see Theorem 
\ref{t:DGH}. In this case we immediately get that the coexistence is impossible, and both populations do not survive. 
%The case $\rho=1$ is very special since 
%the difference between $u,v$ satisfies the discrete heat equation and thus is 
%deterministic. Define

Now, without loss of generality  assume that $\theta_1 \leq  \theta_2$. In what follows, we consider both cases of $d\leq 2$, $b>0$ and $d\geq 3, b>b_\ast$ simultaneously. 
Note that $\eta_t(i)\equiv v_t(i)- u_t(i)\equiv  \theta_2- \theta_1\geq 0$
for all $t\geq 0$ and $i\in \bbZ^d$ since $\eta$ solves the heat equation  starting at constant initial condition $\boldsymbol{\theta}_2-\boldsymbol{\theta}_1$ and thus it remains constant. 
%\begin{equation}
 %\eta_t:=v_t-u_t\,,\quad t\geq0\,.
%\end{equation}
%Then $\dx\eta_t=\Delta \eta_t\dx t$. Let us first consider a special setting 
%and assume that $u_0\leq v_0$ (monotonic initial conditions). In this case, 
%$\eta_t\geq0$ for all $t\geq0$. 
Note that we can re-write the equation for $u$ 
as follows:
\begin{equation}\label{e:pamperturb_a}
 \dx u_t(i)=\Delta u_t(i)\,\dx t+\sqrt{b u_t(i)(u_t(i)+ (\theta_2- \theta_1))}\,\dx W_t(i)\,,\; i\in \bbZ^d.
\end{equation}
%Since for monotonic initial conditions $\eta_t\geq0$ for all $t\geq0$ 
Thus, the 
diffusion coefficient in the generator for $u$ is strictly 
larger than the diffusion coefficient for the PAM$(b)$. 
We can therefore again use a 
comparison result for interacting diffusions
% which is originally due to 
%Cox, Fleischmann, and Greven \cite{CFG96}
% and has been generalised by 
%R\"{u}schendorff, Schnurr, and Wolf \cite{RSW14}
% (the latter version such as in
by Greven, Klenke and Wakolbinger  \cite{GKW02}. 
%~\cite{GKW02} also %
%covers time-inhomogeneous perturbations).
If $w$ is PAM$(b)$ with $w_0=\boldsymbol{\theta}_1$, % the comparison theorem
Theorem 2 in~\cite{GKW02}
 immediately implies
\begin{align}
\bbE_{(\boldsymbol{\theta}_1, \boldsymbol{\theta}_2)}[e^{-\langle u_t,\phi\rangle}]\geq 
\bbE_{\boldsymbol{\theta}_1}[e^{-\langle w_t,\phi\rangle}]\,, \quad \forall t\geq0, 
\end{align}
for any $\phi \in E_\textup{cpt}$. 
By Theorem~\ref{t:DGH}(i), the right hand side of the above equation converges to $1$ as $t\to\infty$. Therefore the 
left hand side also converges to $1$, and thus we get that 
 \begin{equation}
		\langle u_t, \phi\rangle\to 0\,,
		\end{equation}
		in probability as $t\to\infty$. 
Since $v_t=u_t + (\boldsymbol{\theta}_2-\boldsymbol{\theta}_1)$ we immediately get
\begin{equation}
		\langle v_t, \phi\rangle\to (\theta_2-\theta_1)\langle {\boldsymbol 1}, \phi\rangle \,,
		\end{equation}
		in probability as $t\to\infty$, and  this finishes the proof. 
 \qed  
   
   \begin{remark}
Morally, the comparison result 
implies that ``more noise kills''. In other words, if a system of (locally 
1-dimensional) interacting diffusions weakly dies out as time tends to infinity, 
then so does a system which has a larger diffusion coefficient. Intuitively this 
makes sense since locally $0$ is a trap for the system. Very roughly 
speaking, the process gets close to 0 at some time almost surely, and the 
higher the fluctuations the higher the chance that the process will actually 
hit zero once it gets close. Therefore, when initial conditions are 
monotonic, the comparison result implies that the non-coexistence regime of the 
PAM immediately carries over to the SBM. Actually it %does not just 
implies not only  
non-coexistence but % it implies 
also that it is the population which has been smaller 
at the beginning is the one which will suffer extinction.
\end{remark}

\subsection{Proof of Theorem \ref{t:main}(\ref{i:sumnoncoex})}\label{s:noncoexsum}
In the case of monotone initial conditions $v_0(\cdot)\geq u_0(\cdot)$, to get the result, we can argue as in the proof of Theorem \ref{t:main2}(\ref{i:sumnoncoex}) while again using the comparison theorem from~Greven, Klenke and Wakolbinger~\cite{GKW02}. % that  covers time-inhomogeneous perturbations.

Thus, let us consider the general case when initially one population does not necessary dominate another one. Again, in what follows, we consider both cases of $d\leq 2$, $b>0$ and $d\geq 3, b>b_\ast$ simultaneously. 
 Recall that we assume without loss of generality that $\bar u_0=\langle u_0, \boldsymbol{1}\rangle \leq \langle v_0, \boldsymbol{1}\rangle = 
 \bar v_0$.   Define
\begin{equation}
 \eta_t:=v_t-u_t\,,\quad t\geq0\,.
\end{equation}
Then $\dx\eta_t=\Delta \eta_t\dx t, \; t\geq0$. 
%Let us first consider a special setting 
%and assume that $u_0\leq v_0$ (monotonic initial conditions). In this case, 
%$\eta_t\geq0$ for all $t\geq0$.
Note that we can re-write the equation for $u$ 
as follows:
\begin{equation}\label{e:pamperturb}
 \dx u_t(i)=\Delta u_t(i)\,\dx t+\sqrt{b u_t(i)(u_t(i)+\eta_t(i))}\,\dx W_t(i)\,,\quad t\geq0, i\in \bbZ^d. 
\end{equation}

For our general initial conditions in $E_\textup{fin}$ the 
perturbation $\eta_t$ in the diffusion coefficient in \eqref{e:pamperturb} can 
be positive or negative. Hence, the comparison argument used above fails since it is no 
longer true that $u_t(i)(u_t(i)+\eta_t(i))\geq u_t(i)^2$ for all $t\geq0, i\in \bbZ^d$. %We will now deal  with this case.
 Nonetheless, we will again use comparison techniques with 
the PAM. We will quickly explain the main idea.

We will need the following notation: for any $\phi:\bbZ^d\mapsto\bbR$ we denote its 
positive and negative part by
\begin{equation}
 \phi^+:= \max\{\phi,0\}\,,\quad \phi^-:=\max\{-\phi, 0\}\,.
\end{equation}

%What is of course still true for general initial conditions is that $\eta$ 
%solves the heat equation and hence is relatively easy to handle. We suspect 
%that as in the monotonic case the population which has smaller initial total 
%mass, suppose it is $u$, will eventually die out.
 The idea is to decompose 
$u_t$ into the mass where it exceeds $v$ and the minimum of $u$ and $v$. Define 
the minimum process $w=(w_t)_{t\geq0}$ by setting
\begin{equation}\label{e:min}
 w_t(i):=\min\{u_t(i),v_t(i)\}\,,\quad t\geq0\,,  i\in\bbZ^d.   
\end{equation}
Then, clearly
\begin{equation}\label{e:uweta}
 u_t(i)=w_t(i)+\eta_t^-(i)\,,\quad t\geq0\,,  i\in\bbZ^d. 
\end{equation}
%since for $k\in\bbZ^d$, if $u_t(k)\leq v_t(k)$, then 
%$u_t(k)=\min\{u_t(k),v_t(k)\}=w_t(k)$ and if $u_t(k)\geq v_t(k)$, then 
%$u_t(k)= 
%u_t(k)-v_t(k)+v_t(k)=(u_t(k)-v_t(k))+\min\{u_t(k),v_t(k)\}=(v_t(k)-u_t(k))^-+ 
%w_t(k)$.
 We will show that the total mass of the negative part of the 
heat equation, $\eta_t^-$, vanishes as $t\to\infty$. Moreover it turns out that 
the minimum process $w$ exhibits an ``approximate'' duality with the PAM which allows us 
to deduce its extinction from the extinction of the PAM. Then, clearly if the 
excess $u$ has over $v$ vanishes and the minimum of $u$ and $v$ also goes to 0, 
then $u$ will eventually suffer extinction!

Before we proceed to the proof of the non-coexistence result we need a 
brief digression on the heat equation. 

In order not to confuse with the difference process $v_t-u_t$, for the 
following general arguments we use the generic notation $\zeta$ for the 
solution 
of the heat equation. So, for $f:\bbZ^d\mapsto\bbR$, let 
\begin{equation}
  \dx \zeta_f(t,i)=\Delta \zeta_f(t,i)\,\dx t\,,\quad \zeta(0,i)=f(i)\,,\quad 
t\geq0, i\in\bbZ^d\,. 
\end{equation}
If $M>0$ and $f=M\mathbbm{1}_{\{0\}}$, we write $\zeta^M$ for the corresponding solution. In the following, 
for a differentiable function $g:\bbR_+\mapsto\bbR,$ %$t\mapsto g(t)$
 we denote the 
derivative by $\partial_tg$. 

For the proof Theorem~\ref{t:main}(\ref{i:sumnoncoex}) we will
need to understand the behavior of the negative part of $\zeta$. For example, 
we will prove and use the fact that if at time $t=0$ there is an overall 
excess of ``positive heat'', that is, $\langle 
\zeta_f^+,\boldsymbol{1}\rangle\geq 
\langle \zeta_f^-,\boldsymbol{1}\rangle$, the 
negative heat will eventually disappear: $\langle 
\zeta_f^-(t),\boldsymbol{1}\rangle\to 0$, as $t\to\infty$. 
Note that $\langle \zeta_f(t),\boldsymbol{1}\rangle=\langle f,\boldsymbol{1}\rangle$ is constant.

Some of the facts about the discrete heat equation which we will state and 
prove in the following are probably well-known in the literature. 
For completeness, we present anyway these results including our proofs.

For $T>0$ and $i\in\bbZ^d$, let $\mcZ_f(T,i)$ be the set of zeros of 
$\zeta_f(t,i)$ in $[0,T]$, that is, 
\begin{equation}
\mcZ_f(T,i):=\{t\in[0,T]: \zeta_f(t,i)=0\}\,.
\end{equation}
In what follows for any  $A\subset \bbR$, $|A|$ will denote the number of points in $A$.  
\begin{proposition}[$\zeta$ analytic]\label{p:analytic}
 Let $f:\bbZ^d\mapsto\bbR$ be bounded. Then
$\zeta_f(\cdot,i)$ is (real) analytic for each $i\in\bbZ^d$. If $f\not\equiv0$, 
then  $|\mcZ_f(T,i)|<\infty$ for each $T>0,i\in\bbZ^d$.
\end{proposition}

\begin{proof}
In the following we consider $i\in\bbZ^d$ to be fixed. Recall that 
$p_t(\cdot,\cdot)$ denotes the transition probability of a continuous-time, rate $1$,
simple symmetric random walk on $\bbZ^d$ and let $p^n(i,j)$ be the 
probability that a discrete-time simple symmetric random walk on $\bbZ^d$ jumps 
from $i$ to $j$ in $n$ steps. Then, for  $t\geq 0$, 
\begin{align}\label{e:reorder}
 \zeta_f(t,i)
 &=\sum_{j} p_t(i,j) f(j)
 =e^{-t}\sum_{j} \Bigl(\,\sum_{n=0}^\infty p^n(i,j) \frac{t^n}{n!} \,\Bigr) 
f(j)
=e^{-t} \sum_n \frac1{n!} \Bigl(\, \sum_j p^n(i,j) f(j) \,\Bigr) t^n
\end{align}
and thus $\zeta_f(\cdot,i)$ has a representation as a power series whose 
radius of convergence is infinite. Note that if $f$ is bounded, then
\begin{equation}
\sum_n \sum_j  \frac{t^n}{n!} p^n(i,j) |f(j)|  <\infty
\end{equation}
and hence we can re-order the series in \eqref{e:reorder}. 

To prove the second statement, note that it is 
immediate from \eqref{e:reorder} that $\zeta_f(\cdot,i)$ can be 
extended to an analytic function on the complex plane. Assume that for some 
$T>0$, $|\mcZ_f(T,i)|=\infty$. Then $\mcZ_f(T,i)$ has an accumulation point.  
It is well-known that if the zero set of an analytic function has an 
accumulation point then the function is identically zero on its entire domain 
(see for example Theorem~4.3.7 in \cite{MR0447532}). But this is a contradiction 
to the fact that $f\not\equiv0$.
\end{proof}

\begin{corollary}[Differentiability of $\zeta^+$]\label{c:zetadiff}
 Let $f:\bbZ^d\mapsto\bbR$ be bounded, $i\in\bbZ^d$, and $T>0$. There are only 
finitely many points in $[0,T]$ where $\zeta_f^+(\cdot,i)$ is not 
differentiable.
\end{corollary}

\begin{proof}
Clearly the set of points in $[0,T]$ where $\zeta^+(\cdot,i)$ is not 
differentiable is a subset of the set of points in $[0,T]$ where 
$\zeta(\cdot,i)$ changes its sign. But the latter set is a subset of 
$\mcZ_f(T,i)$. Therefore the claim is a consequence of Proposition 
\ref{p:analytic}.
\end{proof}

\begin{proposition}[Properties of $\zeta^-$]\label{p:propofZminus}
 Let $f:\bbZ^d\mapsto\bbR$. There is a function $q_f:\bbR_+\times\bbZ^d\mapsto\bbR$ 
such that 
 \begin{equation}\label{e:equationforq}
  \zeta_f^-(t,i)=f^-(i)+\int_0^t \Delta \zeta_f^-(s,i)\,\dx s- q_f(t,i)\,,\quad 
t\geq0\,, i\in\bbZ^d\,.
 \end{equation}
For each $i\in\bbZ^d$, the function $q_f(\cdot,i):\bbR_+\mapsto\bbR$ %, t\mapsto q_f(t,i)$ 
has the following properties:
\begin{enumerate}[(i)]
 \item $q_f(\cdot,i)$ is non-negative and non-decreasing.
 \item For $T>0$, define 
$\mcQ_f(T,i):=\{t\in[0,T]:\text{the derivative of }q(t,i)\text{ exists in 
}t\}$. Then \mbox{$|T\setminus \mcQ_f(T,i)|<\infty$}. 
 \item For any $t>0$ define 
$$\partial_t 
q_f(t,i):=\left\{\begin{array}{ll} 
\text{{\rm the derivative of} $q(\cdot,i)$ {\rm at} $t$}, &\text{{\rm if} $t\in\bigcup_{T>0} \mcQ_f(T,i)$,
%{\rm for some} $T>t$; 
}\\
\text{$0,$} &\text{\rm if  $t\not\in\bigcup_{T>0} \mcQ_f(T,i).$}
\end{array}
\right.
$$
%and $\partial_t q_f(t,i):=0$ otherwise. 

Then, 
\begin{equation}\label{e:abscont}
 q_f(t,i)=\int_0^t \partial_s q_f(s,i)\,\dx s\,,\quad t\geq0\,,
\end{equation}
that is, $q_f(\cdot,i)$ is absolutely continuous.
\end{enumerate}
\end{proposition}

\noindent
For simplicity we will omit the subindex $f$ whenever confusion is impossible.

\begin{proof}[Proof of Poposition \ref{p:propofZminus}]
%In order to keep notation simple, we 
%will assume $d=1$, but the general case can be proved in exactly the same way. 
Since $f$ is fixed  we drop the subindex $f$ in what follows in the proof of the proposition. Also in the following we consider $i\in\bbZ^d$ 
to be fixed and $\zeta$ as a function of $t$ only. Define 
\begin{equation}\label{e:defQ}
 q(t,i):=-\zeta^+(t,i)+\zeta^+(0,i)+\int_0^t\Delta \zeta^+(s,i)\dx s\,,\quad 
t\geq0. 
%i\in\bbZ^d\,.
\end{equation}
Then, by writing $\zeta^-(t,i)=\zeta^+(t,i)-\zeta(t,i)$ and re-arranging terms 
in the discrete heat equation we obtain 
\begin{align}
 \zeta^-(t,i)&= \zeta^-(0,i)+\int_0^t \Delta \zeta^-(s,i)\,\dx s - q(t,i), \quad 
t\geq0. 
\end{align}
Corollary \ref{c:zetadiff} and the definition of $q$ in \eqref{e:defQ} 
clearly imply that $|T\setminus \mcQ_f(T,i)|<\infty$ for each $T>0$. Since in addition we 
know that $\zeta(\cdot,i)$ is analytic, it is easy to see that 
\eqref{e:abscont} holds true.

We still have to prove that $q(t,i)$ is  non-decreasing in $t$.
As   $(T\setminus\mcQ(T,i))\subset\mcZ_f(T,i)$ and $|\mcZ_f(T,i)|<\infty$ for every $T>0$ it is clearly enough to prove that 
$\partial_t q(t,i)\geq0$ for $t\in\bigcup_{T>0}(T\setminus \mcZ_f(T,i))$. So, let 
$t\in\bigcup_{T>0}(T\setminus \mcZ_f(T,i))$. Then \eqref{e:defQ} implies 
\begin{equation}\label{e:partialq}
 \partial_t q(t,i)=-\partial_t \zeta^+(t,i)+\Delta \zeta^+(t,i)\,.
\end{equation}
Now, assume that $\zeta(t,i)> 0$. Then $ \zeta^+(t,i)= \zeta(t,i)$ and 
$\partial_t \zeta^+(t,i)=\partial_t \zeta(t,i)$. Therefore, since $\partial_t \zeta(t,i)=\Delta \zeta(t,i)$ we get from \eqref{e:partialq} that 
\begin{align}
 \partial_t q(t,i)&=\Delta \zeta^+(t,i)-\Delta \zeta(t,i)\\
 &= \frac{1}{2d} \sum_{j\sim i} \Bigl(\zeta^+(t,j)-\zeta(t,j)\Bigr)\\
 &\geq 0.
\end{align}
%Therefore from \eqref{e:partialq} we obtain
%\begin{equation}
%  \partial_t q(t,i)=\zeta^-(t,i+1)\geq0\,.
%\end{equation}
Alternatively if $\zeta(t,i)< 0$, we have that $\zeta^+(t,i)=0$ and $\partial_t \zeta^+(t,i)=0$. Therefore \eqref{e:partialq} implies
\begin{align}
 \partial_t q(t,i)&=\Delta \zeta^+(t,i)\\
 &= \frac{1}{2d} \sum_{j\sim i}\zeta^+(t,j)\\
 &\geq 0.
\end{align}
The above gives us that $q(t,i)$ is non-decreasing in $t$. 
Since 
$q(0,i)=0$ this implies also $q(t,i)\geq0$. Thus we have proved all claims.
%%%%%%%%%%%%%%%%%%%%%%%%%%%%%%%%%%%%%%%%%%%%%%%%%%%%%%%%%%%%%%%%%%%%
%%% Some more cases worked out %%%%%%%%%%%%%%%%%%%%%%%%%%%%%%%%%%%%%
% The other cases follow the same lines: 
% Now assume that there is $\eps>0$ such that $Z(s,x)\geq0$ and 
% $Z(s,x-1),Z(s,x+1)\leq 0$ for all $s\in[t,t+\eps]$. Therefore, 
% \begin{align}
%  \Delta Z(s,x)
%  &= \frac12 \Bigl(Z(r,x-1)+Z(r,x+1)-2Z(r,x)\Bigr)\\
%  &= \frac12 
% \Bigl(-Z^-(r,x-1)-Z^-(r,x+1)-2Z^+(r,x)+Z^+(s,x-1)+Z^+(s,x+1)\Bigr)\\
%  &= -Z^-(s,x-1)-Z^-(s,x+1)+\Delta Z^+(s,x)
% \end{align}
% so that we obtain
% \begin{align}
%  Q(s,x)-Q(t,x)
%  &=-Z^+(s,x)+Z^+(t,x)+\int_t^s \Delta Z^+(r,x)\,\dx r\\
%  &=-Z(s,x)+Z(t,x)+\int_t^s \Delta Z^+(r,x)\,\dx r\\
%  &=-\int_t^s \Delta Z(r,x)\,\dx r+\int_t^s \Delta Z^+(r,x)\,\dx r\\
%  &=\int_t^s Z^-(r,x-1)+Z^-(x+1)\,\dx r-\int_t^s \Delta Z^+(r,x)\,\dx 
% r+\int_t^s 
% \Delta 
% Z^+(r,x)\,\dx r\\
% &\geq0
% \end{align}
% 
% Now assume that there is $\eps>0$ such that 
% $Z(s,x),Z(s,x-1),Z(s,x+1)\geq 0$ for all $s\in[t,t+\eps]$. Therefore, 
% \begin{align}
%  \Delta Z(s,x)
%  &= \frac12 \Bigl(Z(r,x-1)+Z(r,x+1)-2Z(r,x)\Bigr)\\
%  &= \frac12 \Bigl(Z^+(r,x-1)+Z^+(r,x+1)-2Z^+(r,x)\Bigr)\\
%  &=\Delta Z^+(s,x)
% \end{align}
% so that we obtain
% \begin{align}
%  Q(s,x)-Q(t,x)
%  &=-Z^+(s,x)+Z^+(t,x)+\int_t^s \Delta Z^+(r,x)\,\dx r\\
%  &=-Z(s,x)+Z(t,x)+\int_t^s \Delta Z^+(r,x)\,\dx r\\
%  &=-\int_t^s \Delta Z(r,x)\,\dx r+\int_t^s \Delta Z^+(r,x)\,\dx r\\
%  &=-\int_t^s \Delta Z^+(r,x)\,\dx r+\int_t^s 
% \Delta 
% Z^+(r,x)\,\dx r\\
% &=0
% \end{align}
% etc\dots
%%%%%%%%%%%%%%%%%%%%%%%%%%%%%%%%%%%%%%%%%%%%%%%%%%%%%%%%%%%%%%%%%%%%%%%%%%
\end{proof}

\begin{lemma}\label{l:DHE}
Let $f\in E_\textup{fin}$ and denote $M:=\langle 
f,\boldsymbol{1}\rangle$. Then, 
\begin{equation}\label{2701141124}
 \lim_{t\to\infty} \langle |\zeta_f(t)- \zeta^M(t)|,\boldsymbol{1}\rangle =0\,.
\end{equation} 
\end{lemma}

For a related result for the heat equation on $\mathbb{R}^d$ see (1.10) in 
\cite{EZ91}. We are grateful to Yehuda Pinchover who pointed us to this 
reference.

\begin{proof}
Recall that $p_t(\cdot,\cdot)$ denotes the transition probability of a 
simple symmetric random walk on $\bbZ^d$ with the generator $\Delta$. Note that then $\zeta_f(t,i)=\sum_j 
p_t(i,j)f(j)=\sum_i p_t(0,i-j)f(j)$ and thus 
$\zeta^M(t,i)=Mp_t(0,i)=\sum_j f(j)p_t(0,i)$. Therefore, 
\begin{align}
 \langle |\zeta_f(t)- \zeta^M(t)|,\boldsymbol{1}\rangle
 &=\sum_i \Big|\sum _j  (p_t(0,i-j)-p_t(0,i)) f(j)\Big| \\
 &\leq\sum_j |f(j)|\sum_i| p_t(0,i-j)-p_t(0,i)|, \quad 
t\geq0. 
 \end{align}
 Now since $\sum_i| p_t(0,i-j)-p_t(0,i)| \leq 2$ and $f\in E_\textup{fin}$ the 
dominated convergence theorem implies that it is enough to show that, for each 
$j$, 
\begin{equation}
 \lim_{t\to\infty}\sum_i| p_t(0,i-j)-p_t(0,i)| =0\,.
\end{equation}
Let $p^n(0,i)$ be the probability that a discrete time simple random walk 
on $\bbZ^d$ jumps from $0$ to $i$ in $n$ steps. Then
\begin{align}
 \sum_i| p_t(0,i-j)-p_t(0,i)|
 &\leq e^{-t} \sum_{n\geq0} \frac{t^n}{n!} \sum_i |p^n(0,i-j)-p^n(0,i)|, \quad 
t\geq0. 
\end{align}
By Proposition 2.4.1 in Lawler and Limic \cite{MR2677157}, 
one has for all $j\in\bbZ^d$
\begin{equation}
 \sum_i| p^n(0,i-j)-p^n(0,i)|\leq c|j| n^{-1/2}\,,
\end{equation}
for some constant $c$. Hence,
\begin{equation}
 \sum_i| p_t(0,i-j)-p_t(0,i)| \leq c|j| e^{-t} \sum_{n\geq0} \frac{t^n}{n!}  
\frac1{n^{1/2}}\,,
\end{equation}
which, as can be easily checked, for each $j$ tends to $0$ as $t\to\infty$ and we are done.

%The last convergence follows from the the fact that for arbitrarily large 
%$N\in\bbN$, 
%\begin{align}
 %\lim_{t\to\infty}e^{-t}\sum_{n\geq0} \frac{t^n}{n!} \frac1{n^{1/2}}
% &=\lim_{t\to\infty}e^{-t}\sum_{n<N} \frac{t^n}{n!} \frac1{n^{1/2}}+
% e^{-t}\sum_{n\geq N} \frac{t^n}{n!} \frac1{n^{1/2}}\\
% &= \lim_{t\to\infty}  e^{-t}\sum_{n\geq N} \frac{t^n}{n!} \frac1{n^{1/2}}\\
% &\leq \frac1{N^{1/2}}\,.
%\end{align}
\end{proof}

The above lemma implies the following proposition.

\begin{proposition}[longtime behavior of $\zeta^-$ and $\zeta^+$]\label{l:HE}
Let $f\in E_\textup{fin}$ and denote $M:=\langle 
f,\boldsymbol{1}\rangle$. Assume $M\geq0$. Then,
\begin{align}
\label{1_09_1}
 \lim_{t\to\infty} \langle |\zeta_f(t)|,\boldsymbol{1}\rangle =
 \lim_{t\to\infty} \langle \zeta^+_f(t),\boldsymbol{1}\rangle=M\,,\quad 
 \lim_{t\to\infty} \langle \zeta^-_f(t),\boldsymbol{1}\rangle=0\,.
\end{align}
Moreover, 
\begin{equation}
 \lim_{t\to\infty}\langle q_f(t),\boldsymbol{1}\rangle=\langle 
f^-,\boldsymbol{1}\rangle\,.
\end{equation}
\end{proposition}

\begin{proof}
Recall that $\langle \zeta^M(t),\boldsymbol{1}\rangle=M$ for all $t\geq0$. Thus 
by 
the reverse triangle inequality (applied to the $L^1$ norm on 
$E_\textup{fin}$), 
\begin{align}
 \bigl|\langle |\zeta_f(t)|,\boldsymbol{1}\rangle -M\bigr|&=
 \bigl|\langle |\zeta_f(t)|,\boldsymbol{1}\rangle -\langle 
\zeta^M(t),\boldsymbol{1}\rangle\bigr|
 =\bigl|\|\zeta_f(t)\|_1-\|\zeta^M(t)\|_1\bigr|\\
 &\leq \|\zeta_f(t)-\zeta^M(t)\|_1=\langle 
|\zeta_f(t)-\zeta^M(t)|,\boldsymbol{1}\rangle, \quad 
t\geq0. 
\end{align}
By  Lemma \ref{l:DHE}, $\lim_{t\to\infty}\langle 
|\zeta_f(t)-\zeta^M(t)|,\boldsymbol{1}\rangle=0$.
%For the positive/negative part use that
Now we use %the following
\begin{align}
 \langle \zeta^\pm_f(t),\boldsymbol{1}\rangle
 =\frac12 ( \langle|\zeta_f(t)|,\boldsymbol{1}\rangle\pm \langle 
\zeta_f(t),\boldsymbol{1}\rangle)
 =\frac12(\langle |\zeta_f(t)|,\boldsymbol{1}\rangle\pm M), \quad 
t\geq0. 
\end{align}
to get~\eqref{1_09_1}. 
Finally, \eqref{e:equationforq} and the fact that 
$\lim_{t\to\infty}\langle\zeta^-_f(t),\boldsymbol{1}\rangle= 0$ imply 
$\lim_{t\to\infty}\langle q_f(t),\boldsymbol{1}\rangle= \langle 
f^-,\boldsymbol{1}\rangle$.
\end{proof}

We are now ready to prove the non-coexistence result for SBM(1). 

\begin{proof}[Proof of Theorem \ref{t:main}\eqref{i:sumnoncoex}]
Recall that $\eta_t=v_t-u_t$ and $u_t=w_t+\eta_t^-$, where 
$w_t=\min\{u_t,v_t\}$. Since we assume that $\langle 
u_0,\boldsymbol{1}\rangle\leq \langle 
v_0,\boldsymbol{1}\rangle$, we have 
$\langle \eta_0,\boldsymbol{1}\rangle\geq0$. Therefore, Proposition \ref{l:HE} 
implies that
\begin{equation}\label{e:negtozero}
 \lim_{t\to\infty}\langle \eta_t^-,1 \rangle =0\,.
\end{equation}
Hence, if we can also prove
that 
\begin{equation}\label{e2005141104}
 \wlim_{t\to\infty}\langle w_t,\boldsymbol{1}\rangle=0\,.
\end{equation}
then $\wlim_{t\to\infty}\langle 
u_t,\boldsymbol{1}\rangle=\wlim_{t\to\infty}(\langle w_t,\boldsymbol{1}\rangle 
+\langle\eta_t^-,\boldsymbol{1}\rangle)
=0$ and we are done. So we are left with the task to show 
\eqref{e2005141104}. The idea is to use the fact that $w$ 
is ``approximately'' dual to a PAM as $t\to\infty$.

Note again that $\eta_t=\zeta_{v_0-u_0}(t)$ is  the solution of the discrete heat 
equation started in $v_0-u_0$. Hence, since $w_t=u_t-\eta^-_t$, by Proposition 
\ref{p:propofZminus}, we get that $w$ satisfies the following 
equation
\begin{equation}\label{e:semimtgl}
 w_t(i)=w_0(i)+\int_0^t\Delta w_s(i)\,\dx s+\int_0^t\sqrt{b u_s(i)v_s(i)}\,\dx 
W_s(i)+q(t,i)\,,\quad t\geq0\,,i\in\bbZ^d\,.
\end{equation}
Here, actually $q(t,i)=q_{v_0-u_0}(t,i)$ but for simplicity we omit the 
subindex. By \eqref{e:semimtgl} and the properties of $q$ proved in Proposition 
\ref{p:propofZminus}, $w_{\cdot}(i)$ is a semimartingale, for every $i\in \bbZ^d$. 
Note that 
\begin{equation}\label{e:wsquarelessuv}
  w_t^2(i)\leq u_t(i)v_t(i)\,,\quad t\geq0\,,i\in\bbZ^d\,,
\end{equation}
an observation which will be essential for the comparison between $w$ and the 
PAM. 

Now, fix $\eps>0$ arbitrary small. Recall that we denote $\bar{q}(t)=\sum_i q(t,i)$. By 
Proposition \ref{p:propofZminus}, 
$\lim_{t\to\infty}\bar{q}(t)=\bar{q}^\infty$ exists and is finite.
Define
\begin{equation}
\label{28_12_3}
 T^\ast := \inf\{t\geq 0:0\leq \bar{q}^\infty-\bar{q}(t)\leq \eps\}.
 \end{equation}
%and
%\begin{equation}
%q^\ast(t,i):= q(T^\ast+t,i)-q(T^\ast,i)\,,\quad t\geq0\,,i\in\bbZ^d\,.
%\end{equation}

Now let $\tilde{w}=(\tilde{w}_t)_{t\geq0}$ be a PAM starting at $\tilde w_0$ and independent of $w$.  That 
is, it satisfies 
\begin{align}
\tilde{w}_t(i)= \tilde{w}_0(i)+\int_0^t\Delta\tilde{w}_s(i)\,\dx s
+\sqrt{b \tilde{w}_s(i)^2}\,\dx \tilde{W}_s(i)\,,\quad t\geq0\,, i\in\bbZ^d\,,
\end{align}
where $\tilde{W}(i)$, $i\in\bbZ^d$, are independent Brownian motions which are 
assumed to be also independent of  $\{W(i), i\in\bbZ^d\}$. 
%Denote 
%\begin{equation}
% w_0^\ast(i):=w_0+q(T^\ast,i)\,,\quad i\in\bbZ^d\,.
%\end{equation}

Now recall the following result which is stated in Lemma 4.4.10 of Ethier and Kurtz~\cite{EK86}. 
Let $f:[0,\infty)\times[0,\infty)\mapsto\bbR$ be a function such that 
$f(\cdot,t)$ is absolutely continuous for each $t$ and $f(s,\cdot)$ 
is absolutely continuous for each $s$ and $\int_0^T \int_0^T |
f_1(s,t)|\,\dx s\,\dx t, \int_0^T \int_0^T |f_2(s,t)|\,\dx s\,\dx t<\infty$, $T\geq0$. Here $(f_1,f_2)=\nabla f$. 
%$\partial_s$ and $\partial_t$ 
%denote the partial derivatives with respect to the first and second coordinate, 
%respectively.
 Then, for every $R\geq0$, 
\begin{equation}
\label{eq:2_08_1}
 f(T,0)-f(R,T-R)
 =\int_R^T( f_1(s,T-s)-f_2(s,T-s))\,\dx s\,,\quad\text{for 
almost every $T\geq R$.}
\end{equation}
Lemma 4.4.10 \cite{EK86} actually states this formula only for $R=0$ but it can 
easily verified that it holds for all $R\geq0$.
Now we would like to apply~\eqref{eq:2_08_1} to the function
\begin{equation}
\label{eq:2_08_2}
 f(a,b)=\bbE[\exp(-\langle w_a,\tilde{w}_b\rangle)]\,,\quad a,b\geq0\,,
\end{equation}
with $R=T^\ast$.
Note that It\^{o}'s formula implies that $f$ is absolutely continuous in 
$a$ for each 
fixed $b$ and it is absolutely continuous in $b$ for each fixed $a$. Hence the conditions of  
Lemma 4.4.10 in~\cite{EK86}   are satisfied. Now we use It\^{o}'s formula to calculate $ 
f_1(s,t)$ and $ f_2(s,t)$. Then, for all non-negative $\phi\in E_\textup{fin}$, we have 
\begin{align}
 \exp(-\langle w_s, \phi\rangle)\\
%  &\hspace{-1cm}=\exp(-\langle w_0,\boldsymbol{\theta}\rangle)-
%  \sum_i \int_0^s e^{-\langle 
% w_r,\boldsymbol{\theta}\rangle}\partial_{w(i)}\langle 
% w_r,\boldsymbol{\theta}\rangle\,\dx 
% w_r(i)\\
% &\hspace{-1cm}\quad\,+\frac12 \sum_i\int_0^s e^{-\langle 
% w_r,\tilde{x}\rangle}\partial_{w(i)}\langle 
% w_r,\boldsymbol{\theta}\rangle \partial_{w(i)}\langle 
% w_r,\boldsymbol{\theta}\rangle\,\dx [w(i)]_r\\
&\hspace{-1cm}=\exp(-\langle w_0, \phi\rangle)-
 \sum_i \int_0^s e^{-\langle w_r, \phi\rangle} \phi(i)\,\dx 
w_r(i)\\
&\hspace{-1cm}\quad\,+\frac12 \sum_i\int_0^s e^{-\langle 
w_u,\phi\rangle}  \phi(i)^2\,\dx 
[w(i)]_r\\
&\hspace{-1cm}=\exp(-\langle w_0, \phi\rangle)-
  \int_0^s e^{-\langle w_r, \phi\rangle}\sum_i  \phi(i)\Delta 
w_r(i)\,\dx r\\
&\hspace{-1cm}\quad\,-\int_0^s e^{-\langle w_r,  \phi\rangle} 
\sum_i  \phi(i)\sqrt{bu_r(i)v_r(i)}\,\dx 
W_r(i)
-\sum_i \int_0^s e^{-\langle w_r,  \phi\rangle}  \phi(i)\,\dx 
q(r,i)\\
&\hspace{-1cm}\quad\,+\frac12 \int_0^s e^{-\langle 
w_r, \phi\rangle}\sum_i  \phi(i)^2 bu_r(i)v_r(i)\,\dx 
r, \quad s\geq 0. \label{e:series}
%\\
%&\hspace{-1cm}=\exp(-\langle w_0,\boldsymbol{\theta}\rangle)-
 % \int_0^s e^{-\langle w_r,\boldsymbol{\theta}\rangle}\sum_i \theta\Delta 
%w_r(i)\,\dx r\\
%&\hspace{-1cm}\quad\,-\int_0^s e^{-\langle w_r,\boldsymbol{\theta}\rangle} 
%\sum_i\theta\sqrt{u_r(i)v_r(i)}\,\dx 
%W_r(i)
%-\sum_i \int_0^s e^{-\langle w_r,\boldsymbol{\theta}\rangle} \theta\,\dx 
%q(r,i)\\
%&\hspace{-1cm}\quad\,+ \sum_i \frac12 \int_0^s e^{-\langle 
%w_r,\boldsymbol{\theta}\rangle}\theta^2 b 
%u_r(i)v_r(i)\,\dx r\,.
\end{align}
It is easy to check that $s\mapsto \int_0^s e^{-\langle w_r,  \phi\rangle} 
\sum_i  \phi(i)\sqrt{bu_r(i)v_r(i)}\,\dx 
W_r(i)$ is a martingale and not just the local martingale (see Proposition~3.2 of Blath, D\"{o}ring and  Etheridge~\cite{BDE11} for a relevant result). Thus,
taking expectation on both sides yields
\begin{align}
 \MoveEqLeft\bbE[\exp(-\langle w_s,  \phi\rangle)]\\
 &=\exp(-\langle w_0, \phi\rangle)
 -\bbE\left[\int_0^s e^{-\langle w_r, \phi\rangle}\langle\Delta 
w_r,  \phi\rangle \,\dx 
r\right]
- \bbE\left[\langle \int_0^s e^{-\langle w_r, \phi\rangle} \,\dx q(r), 
\phi\rangle  \right]\\
&\quad\, + \bbE\left[\frac12  \int_0^s e^{-\langle 
w_r, \phi\rangle}\langle bu_rv_r, \phi^2 
\rangle\,\dx r \right], \quad s\geq 0. 
\end{align}
Taking derivative with respect to $s$  (it exits for almost every $s$ for which $\partial_s q(s)$ exists, see Proposition~\ref{p:propofZminus}) yields
\begin{align}
 \partial_s &%\MoveEqLeft
 \bbE\left[\exp(-\langle w_s,  \phi\rangle)\right]\\
 \label{eq:19_09_1}
 &=-\bbE\left[e^{-\langle w_s, \phi\rangle}\langle\Delta 
w_s, \phi \rangle\right]-\bbE\left[e^{-\langle 
w_s, \phi\rangle}\langle 
\partial_s q(s), \phi\rangle\right]+ \bbE\left[\frac12 e^{-\langle 
w_s,  \phi\rangle}\langle bu_sv_s, \phi^2 \rangle\right],\; s\geq 0. 
\end{align}
Now assume that $\tilde{w}_0\in E_\textup{fin}$. In exactly the same way as 
above we calculate for every $\psi\in E_\textup{fin}$ (use that the Laplacian is 
self-adjoint):
\begin{align}
 \partial_t \MoveEqLeft\bbE\left[\exp(-\langle \psi, \tilde{w}_t\rangle)\right]\\
 &=-\bbE\left[e^{-\langle \psi,\tilde{w}_t\rangle}\langle \psi,\Delta 
\tilde{w}_t \rangle\right]
+ \bbE\left[\frac12 e^{-\langle 
\psi,\tilde{w}_t\rangle}\langle \psi^2,b\tilde{w}_t^2 \rangle\right]\\
\label{eq:19_09_2}
&=-\bbE\left[e^{-\langle \psi,\tilde{w}_t\rangle}\langle \Delta 
 \psi,\tilde{w}_t \rangle\right]+ \bbE\left[\frac12 e^{-\langle 
\psi,\tilde{w}_t\rangle}\langle b\psi^2,\tilde{w}_t^2 \rangle\right], \quad t\geq 0.
\end{align}
For $x=(x_1,\ldots, x_d)\in \bbZ^d$ 
define $|x|_{\infty}=\sup_{i=1,\ldots,d}|x_i|$ and $B_n=\{ x\in \bbZ^d, |x|_\infty \leq n\}$. Fix arbitrary $\theta>0$. Let $\tilde w^n_0= \theta \mathbf{1}_{B_n}$ and $\tilde w^n$ be the solution to the PAM starting at $\tilde w^n_0$. 
Now by \eqref{eq:2_08_1}, \eqref{eq:2_08_2}, \eqref{eq:19_09_1} and \eqref{eq:19_09_2}, for any $T>T^*$,  we have  
\begin{align}
 &\bbE\left[\exp(-\langle w_{T^*},\tilde{w}^n_{T-T^*})\right]
 -\bbE\left[\exp(-\langle w_T,\tilde{w}^n_0)\right]\\
 &=\int_{T^*}^T 
 -\bbE\left[e^{-\langle w_s,\tilde{w}^n_{T-s}\rangle}\langle \Delta 
 w_s,\tilde{w}_{T-s} \rangle\right]
 + \bbE\left[\frac12 e^{-\langle w_s,\tilde{w}^n_{T-s}\rangle}\langle bw_s^2,(\tilde{w}^n_{T-s})^2 
\rangle\right]\,\dx s\\
 &\quad\,+\bbE\left[e^{-\langle w_s,\tilde{w}^n_{T-s}\rangle}\langle\Delta 
w_s,\tilde{w}^n_{T-s} \rangle\right]
+\bbE\left[e^{-\langle w_s,\tilde{w}^n_{T-s}\rangle}\langle 
\partial_s q(s),\tilde{w}^n_{T-s}\rangle\right]\\
&\quad\,- \bbE\left[\frac12 e^{-\langle w_s,\tilde{w}^n_{T-s}\rangle}\langle bu_sv_s,(\tilde{w}^n_{T-s})^2 
\rangle\right]\\
&=\int_{T^*}^T  \bbE\left[\frac12 e^{-\langle 
w_s,\tilde{w}^n_{T-s}\rangle}\langle b(w_s^2-u_sv_s),(\tilde{w}^n_{T-s})^2 \rangle\right]\,\dx s
+\bbE\left[e^{-\langle w_s,\tilde{w}^n_{T-s}\rangle}\langle 
\partial_s q(s),\tilde{w}^n_{T-s}\rangle\right] 
\\
&\leq \int_0^T \bbE\left[e^{-\langle w_s,\tilde{w}^n_{T-s}\rangle}\langle 
\partial_s q(s),\tilde{w}^n_{T-s}\rangle\right]\,\dx s
\label{e:wsquarelessuv2}
\\
&\leq \int_{T^*}^T \langle \partial_s q(s)	,\bbE\left[\tilde{w}^n_{T-s}\right]\rangle\,\dx 
s
\\
\label{eq:19_09_3}
&=\int_{T^*}^T \langle \partial_s q(s)	,P_{T-s}\tilde{w}^n_{0}\rangle\,\dx 
s
\,.
\end{align}
where we have used \eqref{e:wsquarelessuv} in \eqref{e:wsquarelessuv2} and also the fact that $\partial_s q(s,i)\geq 0,$ for all $ i\in \bbZ^d, t>0 $.
It is easy  to see that as $n\rightarrow\infty$,  $\tilde w^n$ converges weakly to  $\tilde w$, where $\tilde w$ is the solution to PAM with initial condition
 $\tilde{w}_0=\boldsymbol{\theta}$. Also, by monotone convergence theorem, we get  
\begin{equation}
P_{T-s}\tilde{w}^n_{0}(i)  \rightarrow \theta,\;\forall i\in\bbZ^d. 
\end{equation}
as $n\to\infty$. 
Moreover the above convergence of $P_{T-s}\tilde{w}^n_{0}(\cdot)$ to $\theta$ is bounded pointwise convergence and thus we get,
\begin{equation}
\label{eq:19_09_4}
\int_{T^*}^T \langle \partial_s q(s)	,P_{T-s}\tilde{w}^n_{0}\rangle\,\dx s \to 
\theta\int_{T^*}^T \langle \partial_s q(s), \mathbf{1}	 \rangle\,\dx s,\;\text{
as $n\to\infty$. }
\end{equation}

Thus, by~\eqref{eq:19_09_3}, \eqref{eq:19_09_4} and weak convergence of  $\tilde w^n$ to $\tilde w$, we get
\begin{align}
 &\bbE\left[\exp(-\langle w_{T^*},\tilde w_{T-T^*})\right]
 -\bbE\left[\exp(-\theta\langle w_T,\mathbf{1})\right]\\
&\leq \theta\int_{T^*}^T \langle \partial_s q(s), \mathbf{1}	 \rangle\,\dx s  %\,, T\geq T^*. 
%\label{28_12_4}
\\
&=\theta (\bar q(T)-\bar q(T^*)) 
\\
%&
% =\theta q^\ast(T-T^*)\\
\label{28_12_6}
& \leq \eps\theta, \;\forall  T\geq T^\ast,
\end{align}
where the last inequality follows by~\eqref{28_12_3}. 
 %Lemma 2.3 of 
%\cite{GDH07} implies $\bbE[\tilde{w}_t(i)]=\theta$.
%Therefore,
%\begin{align}
% |\int_{T^*}^T \bbE[\langle \partial_s q(s),\tilde{w}_{T-s}\rangle]\,\dx s|
 %&\leq |\int_{T^*}^T \langle \partial_s q(s),\bbE[\tilde{w}_{T-s}]\rangle\,\dx s|
 %=\theta|\int_{T^*}^T \langle \partial_s q(s),\boldsymbol{1}\rangle\,\dx s|\\
 %&=\theta \langle \bar q(T)-\bar q(T^*),\boldsymbol{1}\rangle 
%\\
%&
% =\theta q^\ast(T-T^*)\\
%\label{28_12_6}
%& \leq \eps\theta, \;\forall  T\geq T^\ast.  
%\end{align}
We know from Theorem \ref{t:DGH} that  for any compactly supported function $\phi$, 
$\wlim_{t\to\infty}\langle \phi,\tilde{w}_t\rangle=0$ if $b>b_\ast$. Note that $ w_{T^*}$ is not necessarily compactly supported, so we need an additional simple argument. Fix $\tilde \epsilon>0 $ arbitrary small. Since $P_{T^\ast}w_0\in  E_\textup{fin}$ we can take $n$ sufficiently large such that 
\begin{align}
\bbE\left[  \langle w_{T^*}\mathbf{1}_{B^c_n},\tilde w_{T-T^*}\rangle \right]&= \theta  \langle (P_{T^*}w_{0}) \mathbf{1}_{B^c_n},\mathbf{1}\rangle\\
 &= \theta  \langle P_{T^*}w_{0}, \mathbf{1}_{B^c_n}\rangle \\
\label{28_12_1} &\leq \tilde\epsilon. 
\end{align} 
On the other hand, by Theorem \ref{t:DGH},  
\begin{equation}
\label{28_12_2}
\langle w_{T^*}\mathbf{1}_{B_n},\tilde w_{T-T^*}\rangle \Rightarrow 0,
\end{equation} 
at $T\rightarrow\infty$. Since $\tilde\epsilon$ was arbitrary small we get from \eqref{28_12_1}, \eqref{28_12_2}
\begin{equation}
\langle w_{T^*},\tilde w_{T-T^*}\rangle \Rightarrow 0,
\end{equation} 
at $T\rightarrow\infty$.
Hence $\lim_{T\to\infty} \bbE[\exp(-\langle 
w_{T*},\tilde{w}_{T-T^*}\rangle)]=1$. So we 
can find
$T^{\ast\ast}>T^\ast$ such that 
\begin{equation}
\label{28_12_5}
| \bbE[\exp(-\langle 
w_{T*},\tilde{w}_{T-T^*}\rangle)]-1|\leq \eps,\;\text{for}\;T>T^{\ast\ast}.
\end{equation} Hence for $t>T^{\ast\ast}$, we get by~%\eqref{28_12_4},
 \eqref{28_12_6}, \eqref{28_12_5}
\begin{align}
 \bbE[\exp(-\theta\langle w_t,\boldsymbol{1}\rangle)]\geq 
1-\eps(1+\theta).  
\end{align}
This implies \eqref{e2005141104} since $\eps$ was arbitrary small. 
\end{proof}

\bibliography{leonidpatric}
\bibliographystyle{alpha}

\end{document}